\documentclass[12pt]{amsart}
\usepackage{amssymb}
\usepackage{amsmath}
\usepackage{amscd}
\usepackage[all]{xy}
\usepackage{ytableau}
\usepackage{yhmath}
\usepackage{colortbl}
\usepackage{arydshln}
\theoremstyle{plain}
\newtheorem{theorem}{Theorem}[section]
\newtheorem{corollary}[theorem]{Corollary}

\newtheorem{lemma}[theorem]{Lemma}
\newtheorem{proposition}[theorem]{Proposition}

\theoremstyle{definition}
\newtheorem{definition}[theorem]{Definition}

\newtheorem{example}[theorem]{Example}

\theoremstyle{remark}
\newtheorem{remark}[theorem]{Remark}

\numberwithin{equation}{section}

\def\Z{{\mathbb Z}}
\def\N{{\mathbb N}}
\def\P{{\mathbb P}}

\def\C{{\mathbb C}}
\def\Q{{\mathbb Q}}

\def\fka{{\mathfrak a}}
\def\fkb{{\mathfrak b}}

\def\fkg{{\mathfrak g}}
\def\fkh{{\mathfrak h}}

\def\fkl{{\mathfrak l}}

\def\fkn{{\mathfrak n}}
\def\fko{{\mathfrak o}}
\def\fkp{{\mathfrak p}}

\def\fks{{\mathfrak s}}

\def\fkz{{\mathfrak z}}

\def\a{{\boldsymbol a}}

\def\e{{\boldsymbol e}}

\def\m{{\boldsymbol m}}

\def\x{{\boldsymbol x}}
\def\y{{\boldsymbol y}}

\def\0{{\boldsymbol 0}}
\def\1{{\boldsymbol 1}}

\def\ad{{\rm ad}}
\def\Ad{{\rm Ad}}

\def\Gr{{\rm Gr}}

\def\diag{{\rm diag}}
\def\hite{{\rm ht}}

\makeatletter
\newif\if@borderstar
\def\bordermatrix{\@ifnextchar*{%
 \@borderstartrue\@bordermatrix@i}{\@borderstarfalse\@bordermatrix@i*}%
}
\def\@bordermatrix@i*{\@ifnextchar[{\@bordermatrix@ii}{\@bordermatrix@ii[()]}}
\def\@bordermatrix@ii[#1]#2{%
\begingroup
 \m@th\@tempdima8.75\p@\setbox\z@\vbox{%
 \def\cr{\crcr\noalign{\kern 2\p@\global\let\cr\endline }}%
 \ialign {$##$\hfil\kern 2\p@\kern\@tempdima & \thinspace %
  \hfil $##$\hfil && \quad\hfil $##$\hfil\crcr\omit\strut %
  \hfil\crcr\noalign{\kern -\baselineskip}#2\crcr\omit %
  \strut\cr}}%
 \setbox\tw@\vbox{\unvcopy\z@\global\setbox\@ne\lastbox}%
 \setbox\tw@\hbox{\unhbox\@ne\unskip\global\setbox\@ne\lastbox}%
 \setbox\tw@\hbox{%
  $\kern\wd\@ne\kern -\@tempdima\left\@firstoftwo#1%
  \if@borderstar\kern 2pt\else\kern -\wd\@ne\fi%
 \global\setbox\@ne\vbox{\box\@ne\if@borderstar\else\kern 2\p@\fi}%
 \vcenter{\if@borderstar\else\kern -\ht\@ne\fi%
  \unvbox\z@\kern -\if@borderstar2\fi\baselineskip}%
 \if@borderstar\kern-2\@tempdima\kern2\p@\else\,\fi\right\@secondoftwo#1 $%
 }\null \;\vbox{\kern\ht\@ne\box\tw@}%
\endgroup
}
\def\iddots{\mathinner{\mkern1mu\raise\p@
    \hbox{.}\mkern2mu\raise4\p@\hbox{.}\mkern2mu
    \raise7\p@\vbox{\kern7\p@\hbox{.}}\mkern1mu}}
\makeatother

\begin{document}

\title[Jordan Lie
subalgebras]{Limits of Jordan Lie
subalgebras}
\author{Mutsumi Saito}

\address{Department of Mathematics,
Graduate School of Science,
Hokkaido University,
Sapporo, 060-0810, Japan
}

\begin{abstract}
Let $\fkg$ be a simple Lie algebra of rank $n$ over $\C$.
We show that the $n$-dimensional abelian ideals 
of a
Borel subalgebra of $\fkg$ are limits of Jordan Lie subalgebras.
Combining this with a classical result by Kostant, 
we show that the $\fkg$-module spanned by
all $n$-dimensional abelian Lie subalgebras of $\fkg$
is actually
spanned by the Jordan Lie subalgebras.
\end{abstract}



\maketitle

\section{Introduction}

To define a system for generalized Airy functions,
Gel'fand, Retahk, and Serganova \cite{Gelfand-Retahk-Serganova}
considered a Jordan group, which is the centralizer of a maximal 
Jordan cell in $GL(n)$.
We call its Lie algebra a Jordan Lie subalgebra.
The system is a confluent version of an Aomoto-Gel'fand system
(\cite{Aomoto}, \cite{Gelfand}, etc.) associated to a Cartan subalgebra
of $\fkg\fkl_n$.
Kimura and Takano \cite{Kimura-Takano}
explained the process of confluence by taking limits of regular elements;
a Cartan subalgebra is the centralizer of a semisimple regular
element, and a Jordan Lie subalgebra is that of a nilpotent
regular element.
A natural question thus arises;
describe the set of limits of Cartan subalgebras.
Recall that an element $X$ in a simple Lie algebra $\fkg$ is 
said to be regular if the centralizer 
$\fkz_\fkg(X)$ has the
minimal possible dimension, the rank of $\fkg$.

Let $\fkg$ be a simple Lie algebra of rank $n$ over $\C$,
and $G$ its adjoint group.
Let $\fkh$ be a Cartan subalgebra of $\fkg$.
The question is to consider the closure
$
\overline{\Ad(G)\fkh}
$
in the Grassmannian $\Gr(n, \fkg)$ composed of $n$-dimensional subspaces of $\fkg$.
The centralizer of a regular element certainly belongs to
$\overline{\Ad(G)\fkh}$.
In particular, a Jordan Lie subalgebra $J$ that is the centralizer
of a regular nilpotent element belongs to $\overline{\Ad(G)\fkh}$.
As a generalization of a regular nilpotent element,
Ginzburg \cite{Ginzburg} defined and studied a principal nilpotent pair
(also see \cite{Elashvili-Panyushev}).
We remark that 
its centralizer also belongs to $\overline{\Ad(G)\fkh}$.
More generally a wonderful nilpotent pair was studied in
\cite{PanyushevIMRN} and \cite{Yu};
the $\Z_{\geq 0}\times\Z_{\geq 0}$-graded component of its centralizer 
belongs to $\overline{\Ad(G)\fkh}$.


The theory of abelian ideals of a Borel subalgebra
draw a strong attention to researchers in the representation theory
(e.g., \cite{Cellini-Papi},\cite{Panyushev})
after Kostant's remarkable paper \cite{Kostant1998}
related the theory to the combinatorics of affine Weyl groups
and the theory of discrete series.

In this paper, we show that the $n$-dimensional abelian ideals 
of a
Borel subalgebra of $\fkg$ belong to
$\overline{\Ad(G)J}$.
Combining this with a classical result by Kostant \cite{Kostant1965}, 
we show that the $\fkg$-module spanned by
all $n$-dimensional abelian Lie subalgebras of $\fkg$
is actually
spanned by the Jordan Lie subalgebras.

In Section 2, 
after reviewing the classical result by Kostant \cite{Kostant1965},
we state the main results in this paper.
Then we introduce two types of deformation in Section 3:
unipotent deformation and semisimple deformation.
These are two basic techniques we employ.

In the subsequent sections, we prove the results type by type.
For the classical types, we move some technical details
into Appendix, to make the proofs clearer.

In Section 4, we treat the case of type $A$.
Using the Weyl group action,
we reduce the proof to a problem of the solvability of
a system of inequalities, which is proved in Appendix A.
In Section 5, we consider the other classical types in a uniform manner.
In Sections 6 through 10, we treat the exceptional types.
To compute, we fix a Chevalley basis of $\fkg$ 
as in \cite[Proposition 4]{Kurtzke}.

\section{Main results}

Let $\fkg$ be a simple complex Lie algebra of rank $n$, and
$G$ its adjoint group.

\subsection{Kostant's classical result}

For $k=0,1,\ldots, \dim \fkg$,
$\bigwedge^k \fkg$ is a $\fkg$-module by the adjoint representation.
Let $C_k$ be the subspace of $\bigwedge^k \fkg$ spanned by all $\wedge^k \fka$
where $\fka$ is a $k$-dimensional abelian Lie subalgebra of $\fkg$.
Then $C_k$ is a $\fkg$-submodule of $\bigwedge^k \fkg$.

Fix a Cartan subalgebra $\fkh$ and a Borel subalgebra $\fkb\supseteq \fkh$.
Let $\Delta$ be the root system with respect to $\fkh$, and $\Delta^+$
the positive root system corresponding to $\fkb$.
As a $\fkg$-module, $C_k$ is characterized by the following theorem:

\begin{theorem}[Kostant \cite{Kostant1965}]
\label{thm:Kostant65}
Let $\fka$ be a $k$-dimensional abelian ideal of $\fkb$.
Then $\wedge^k \fka$ is a highest weight vector of $C_k$.
Conversely any highest weight vector of $C_k$ is of this form.
\end{theorem}

Let $\fka$ be an abelian ideal of $\fkb$.
Note that there exists a subset $\Delta(\fka)\subseteq \Delta^+$
such that
\begin{equation}
\label{eqn:DeltaA}
\text{$\fka=\bigoplus_{\alpha\in \Delta(\fka)}\fkg_\alpha$ \quad and \quad
$(\Delta^+ +\Delta(\fka))\cap \Delta\subseteq \Delta(\fka)$.}
\end{equation}

\subsection{Jordan Lie subalgebras}

Let $\alpha_1,\ldots,\alpha_n$ be the simple roots in $\Delta^+$;
we follow Bourbaki's notation \cite{Bourbaki}.

Let 
$$
\{ X_\alpha, H_i\,|\, \alpha\in \Delta,\, i=1,2,\ldots,n\}
$$
be a Chevalley basis 
of $\fkg$.
Let $\Lambda:=\sum_{i=1}^n X_{\alpha_i}$,
and $J:=\fkz_{\fkg}(\Lambda)$.
Then $\Lambda$ is a regular nilpotent element
(cf. \cite[Theorem 5.3]{Kostant1959}),
and $J$ is called a Jordan Lie subalgebra of $\fkg$.

We have the following proposition
(see \cite[Lemma 2.5]{Ranee} and \cite[(1.6)]{Ginzburg}):

\begin{proposition}
\label{prop:Ranee}
In the Grassmannian $\Gr(n, \fkg)$,
$$
J=\lim_{t\to 0}\exp t^{-1}\ad \Lambda (\fkh)
\in \overline{\Ad(G)\fkh}.
$$
\end{proposition}

For $\alpha\in \Delta^+$, let $\hite(\alpha)$ denote the height of $\alpha$.
Then the nilradical $\fkn$ of $\fkb$ is graded by $\hite$:
$$
\fkn=\bigoplus_{j>0} \fkg_j
\qquad
\fkg_j:=\bigoplus_{\hite(\alpha)=j}\fkg_\alpha.
$$
The Jordan Lie subalgebra $J=\fkz_\fkg(\Lambda)$ is also graded by $\hite$:
$$
J=\bigoplus_j J\cap \fkg_j.
$$
The set of heights appearing in $J$ is exactly the same as 
that of exponents of $\fkg$ counting multiplicities
(cf. \cite[Theorem 6.7]{Kostant1959}).

In the following classical examples,
we take the subset of diagonal matrices and that of upper
triangular matrices as $\fkh$ and $\fkb$, respectively.
Let $\varepsilon_i\in \fkh^*$ denote the linear form
taking the $(i,i)$-component.
The Jordan Lie subalgebras $J$ below can be computed as follows:
First it is easy to check that $Z$ and $\Lambda^i$ belong to 
$\fkg$ for the indicated
powers $i$. It is also clear that they commute with $\Lambda$.
Since we know the heights appearing in $J$
(\cite[Theorem 6.7]{Kostant1959} loc. cit.), we see that 
they form a basis of $J$.

We denote by $E_{i,j}$ the matrix whose entries are $0$ except for the
$(i,j)$-entry $1$.
Let $\gamma_0$ denote the maximal root.

\begin{example}
\label{A-1}
Let $\fkg=\fks\fkl(n+1, \C)$, and
$X_{\alpha_i}:=E_{i, i+1}$ $(1\leq i \leq n)$.
Then $\Lambda=\sum_{i=1}^{n}E_{i, i+1}$,
 and
$$
J=\bigoplus_{i=1}^n \C \Lambda^i.
$$
We have $\gamma_0=\sum_{i=1}^n\alpha_i$
and $\hite(\gamma_0)=n$.
\end{example}

\begin{example}
\label{B-1}
Let 
$\displaystyle
F:=
\begin{bmatrix}
0 & & 1\\
 & \adots & \\
1 && 0
\end{bmatrix},
$
and let
\begin{eqnarray*}
\fkg:=
\fks\fko(2n+1, \C)
&=&
\left\{
X\in \fks\fkl(2n+1)\,|\,
{}^tXF+FX=O
\right\}
\\
&=&
\left\{
\begin{bmatrix}
A & \x & B\\
-{}^t\y & 0 & -{}^t\x \\
C & \y & -A'
\end{bmatrix}
\,|\,
B'=-B,\, C'=-C
\right\},
\end{eqnarray*}
where $A,B,C$ are $n\times n$ matrices,
$\x, \y$ are column vectors of dimension $n$, and
\begin{equation}
\label{antitranspose}
\text{$A':=(a_{n+1-j,n+1-i})$
for an $n\times n$ matrix $A=(a_{i,j})$.}
\end{equation}
The simple roots are
$$
\alpha_1=\varepsilon_1-\varepsilon_2,\ldots,
\alpha_{n-1}=\varepsilon_{n-1}-\varepsilon_n,
\alpha_n=\varepsilon_n.
$$
Let
$X_{\alpha_i}:=E_{i, i+1}-E_{2n+1-i,2n+2-i}$ $(1\leq i \leq n)$.

Then
$$
\Lambda=
\sum_{i=1}^{n} E_{i, i+1} -\sum_{i=n+1}^{2n}E_{i,i+1},
$$
and
$$
J=\bigoplus_{k=1}^{n}
\C \Lambda^{2k-1}.
$$


We have 
$\gamma_0=\varepsilon_1+\varepsilon_2=\alpha_1+2\sum_{k=2}^n\alpha_k$
and
$\hite(\gamma_0)=2n-1$.
\end{example}

\begin{example}
\label{C-1}
Let 
$\displaystyle
F:=
\begin{bmatrix}
0 & & & & &1\\
&&& & \adots & \\
&&& 1 &&\\
&& -1 &&&\\
& \adots &&&&\\
-1 &&&&& 0
\end{bmatrix},
$
and let
\begin{eqnarray*}
\fkg:=
\fks\fkp(2n, \C)
&=&
\left\{
X\in \fks\fkl(2n)\,|\,
{}^tXF+FX=O
\right\}
\\
&=&\left\{
\begin{bmatrix}
A & B\\
C & -A'
\end{bmatrix}
\,|\,
B'=B,\, C'=C
\right\},
\end{eqnarray*}
where $A,B,C$ are $n\times n$ matrices
(cf. \eqref{antitranspose} for $A'$ etc.).

The simple roots are
$$
\alpha_1=\varepsilon_1-\varepsilon_2,\ldots,
\alpha_{n-1}=\varepsilon_{n-1}-\varepsilon_n,
\alpha_n=2\varepsilon_n.
$$
Let $X_{\alpha_i}:= 
E_{i, i+1}-E_{2n-i,2n+1-i}$ $(1\leq i \leq n-1)$,
and
$X_{\alpha_n}:=E_{n,n+1}$.

Then $\Lambda=
\sum_{i=1}^n E_{i, i+1} -\sum_{i=n+1}^{2n-1}E_{i,i+1}
$, and
$$
J=\bigoplus_{k=1}^{n}
\C \Lambda^{2k-1}.
$$
We have $\gamma_0=2\varepsilon_1=
2\sum_{i=1}^{n-1}\alpha_i +\alpha_n$
and
$\hite(\gamma_0)=2n-1$.
\end{example}

\begin{example}
\label{D-1}
Let $\displaystyle
F:=
\begin{bmatrix}
0 & & 1\\
 & \adots & \\
1 && 0
\end{bmatrix},
$
and let
\begin{eqnarray*}
\fkg:=
\fks\fko(2n, \C)
&=&
\left\{
X\in \fks\fkl(2n)\,|\,
{}^tXF+FX=O
\right\}
\\
&=&\left\{
\begin{bmatrix}
A & B\\
C & -A'
\end{bmatrix}
\,|\,
B'=-B,\, C'=-C
\right\},
\end{eqnarray*}
where $A,B,C$ are $n\times n$ matrices.

The simple roots are
$$
\alpha_1=\varepsilon_1-\varepsilon_2,\ldots,
\alpha_{n-1}=\varepsilon_{n-1}-\varepsilon_n,
\alpha_n=\varepsilon_{n-1}+\varepsilon_n.
$$
Let
$X_{\alpha_i}:=E_{i,i+1}-E_{2n-i,2n+1-i}$ $(1\leq i\leq n-1)$
and
$X_{\alpha_n}:=E_{n-1,n+1}-E_{n,n+2}$.
Then
\begin{eqnarray*}
\Lambda&=&
\sum_{i=1}^{n-1} E_{i, i+1} -\sum_{i=n+1}^{2n-1}E_{i,i+1}
+E_{n-1, n+1}-E_{n,n+2},
\end{eqnarray*}
and
$$
J=\C Z\bigoplus\bigoplus_{k=1}^{n-1}
\C \Lambda^{2k-1},
$$
where 
$Z=E_{1,n}-E_{n+1,2n}-E_{1,n+1}+E_{n,2n}$.

The height of $\Lambda^{2k-1}$ equals $2k-1$, and
that of $Z$ $n-1$.
We have
$\gamma_0=\varepsilon_1+\varepsilon_2
=\alpha_1+2\sum_{k=1}^{n-2}\alpha_k+\alpha_{n-1}+\alpha_n$
and
$\hite(\gamma_0)=2n-3$.
\end{example}

\begin{proposition}
\label{prop:JtoK}
There exists an abelian Lie subalgebra 
$K \in \overline{\Ad(G)J}$
with a basis 
$\{ \Lambda^{(i)}\,|\, \hite(\gamma_0)-(n-1)\leq i\leq \hite(\gamma_0)\}$
unless $\fkg$ is of type $D_n$, and
with a basis
$\{ \Lambda^{(i)}\,|\, \hite(\gamma_0)-(n-2)\leq i\leq \hite(\gamma_0)\}\cup
\{ Z\}$
when $\fkg$ is of type $D_n$,
where
$\Lambda^{(i)}$ is 
of the following form:
$$
\Lambda^{(i)}=\sum_{\hite(\alpha)=i}c_\alpha X_\alpha
\qquad
(\text{$c_\alpha\neq 0$ for any $\alpha$}).
$$
Furthermore, in the case of type $D_n$,
we can take $\Lambda^{(n-1)}$
so that $Z$ and
$c_{\varepsilon_1-\varepsilon_n}X_{\varepsilon_1-\varepsilon_n}
+c_{\varepsilon_1+\varepsilon_n}X_{\varepsilon_1+\varepsilon_n}$
are
linearly independent.
\end{proposition}

\begin{proof}
\noindent
When $\fkg$ is a simple Lie algebra of classical type,
there exists a nilpotent element $S\in \fkg$ such that
$K:=\lim_{t\to 0}\exp(t^{-1}\ad S)(J)$.
Indeed, if $\fkg$ is of type $A$, then
take $S=0$. Then $K=J$ and $\Lambda^{(i)}=\Lambda^i$ meet
the condition.
For the other classical types, we prove the existence of such an $S$ 
in Appendix.

For a simple Lie algebra of exceptional type,
we see the statement in its own section.
\end{proof}

\subsection{Main Theorem}

The following is the main theorem of this paper.
The proof is given by a type-by-type consideration.
Propositions \ref{prop:Ranee} and \ref{prop:JtoK}
lead to the latter half of the statements in Theorem \ref{MainTheorem}.

\begin{theorem}
\label{MainTheorem}
Let $\fka$ be an $n$-dimensional abelian ideal of $\fkb$,
and let $K$ be an abelian Lie subalgebra described in Proposition
\ref{prop:JtoK}.
Then $\fka\in \overline{\Ad(G) K}$ in $\Gr(n, \fkg)$.
Hence
$\fka\in \overline{\Ad(G) J}$, and
$\fka\in \overline{\Ad(G) \fkh}$.
\end{theorem}

\begin{corollary}
\label{cor:Main}
The subspace $C_n$ of $\wedge^n \fkg$
is spanned by any of the following:
\begin{enumerate}
\item[{\rm (1)}]
$\{\wedge^n \fka\,|\, \fka\in {\Ad(G) K}\}$,
\item[{\rm (2)}]
$\{\wedge^n \fka\,|\, \fka\in \overline{\Ad(G) K}\}$,
\item[{\rm (3)}]
$\{\wedge^n \fka\,|\, \fka\in {\Ad(G) J}\}$,
\item[{\rm (4)}]
$\{\wedge^n \fka\,|\, \fka\in \overline{\Ad(G) J}\}$,
\item[{\rm (5)}]
$\{\wedge^n \fka\,|\, \fka\in {\Ad(G) \fkh}\}$,
\item[{\rm (6)}]
$\{\wedge^n \fka\,|\, \fka\in \overline{\Ad(G) \fkh}\}$.
\end{enumerate}
\end{corollary}

\begin{proof}
It is enough to prove (1).
The subspace spanned by any set from (1) to (6)
is a $G$-submodule of $C_n$, and thus a $\fkg$-submodule.
Hence (2) is clear from Theorems \ref{thm:Kostant65}
and \ref{MainTheorem}.
Let $C_n'$ be the subspace spanned by 
$\{\wedge^n \fka\,|\, \fka\in {\Ad(G) K}\}$.
Since this is closed and includes 
$\{\wedge^n \fka\,|\, \fka\in {\Ad(G) K}\}$,
we see
$$
C_n'\supseteq \{\wedge^n \fka\,|\, \fka\in \overline{\Ad(G) K}\}.
$$
Hence we obtain (1) from (2).
\end{proof}

We use the following lemma very often throughout this paper.
\begin{lemma}
\label{lemma:ElementVsSpace}
Let $\C^\times\ni t\mapsto \fka_t\in \Gr(n, \fkg)$
and $\C^\times \ni t\mapsto A_t\in \fkg\setminus\{ 0\}$
be morphisms.
Suppose that $A_t\in \fka_t$ for all $t\in \C^\times$,
and $A:=\lim_{t\to 0}A_t$ exists in $\fkg\setminus\{ 0\}$.

Then $A\in \lim_{t\to 0}\fka_t$.
\end{lemma}

\begin{proof}
Consider morphisms
\begin{eqnarray*}
P&:&\fkg\times (\fkg)^n\ni
(Y, [\a_1,\ldots,\a_n])\mapsto
Y\wedge \a_1\wedge\cdots\wedge\a_n\in \bigwedge^{n+1}\fkg,\\
P'&:&(\fkg)^n\ni
[\a_1,\ldots,\a_n]\mapsto
\a_1\wedge\cdots\wedge\a_n\in \bigwedge^{n}\fkg.
\end{eqnarray*}
Then $P^{-1}(0)$ ($P'^{-1}(0)$ respectively)
is closed and $\C^\times \times GL(n)$-stable
($GL(n)$-stable respectively).
Hence $P^{-1}(0)\cap ((\fkg\setminus\{ 0\})\times(\fkg^n\setminus P'^{-1}(0)))$
is closed in $(\fkg\setminus\{ 0\})\times(\fkg^n\setminus P'^{-1}(0))$
and $\C^\times \times GL(n)$-stable.

Thus its image 
$$
\{ (\C Y, \fka)\, |\, Y\in \fka\}
$$
under the canonical morphism
is closed in  $\P(\fkg)\times \Gr(n, \fkg)$.
Hence 
$(\C A, \lim_{t\to 0}\fka_t)=\lim_{t\to 0}(\C A_t, \fka_t)$
belongs to $\{ (\C Y, \fka)\, |\, Y\in \fka\}$, i.e.,
$A\in \lim_{t\to 0}\fka_t$.
\end{proof}

We close this section with the following small example of 
Proposition \ref{prop:JtoK} and Theorem
\ref{MainTheorem}:

\begin{example}\rm
Let $\fkg=\fks\fkp(6,\C)$.
As in Example \ref{C-1}, let
$$
J=\left\{ A=a\Lambda+b\Lambda^3+c\Lambda^5=
\begin{bmatrix}
0 & a & 0 & b & 0 & c\\
0 & 0 & a & 0 & -b & 0\\
0 & 0 & 0 & a & 0 & b\\
0 & 0 & 0 & 0 & -a & 0\\
0 & 0 & 0 & 0 & 0 & -a\\
0 & 0 & 0 & 0 & 0 & 0
\end{bmatrix}
\right\}.
$$
There are two $3$-dimensional abelian ideals
of the upper triangular Borel subalgebra:
$$
\fka_1=\left\{
\begin{bmatrix}
0 & 0 & 0 & b & a & c\\
0 & 0 & 0 & 0 & 0 & a\\
0 & 0 & 0 & 0 & 0 & b\\
0 & 0 & 0 & 0 & 0 & 0\\
0 & 0 & 0 & 0 & 0 & 0\\
0 & 0 & 0 & 0 & 0 & 0
\end{bmatrix}
\right\},\quad
\fka_2=\left\{
\begin{bmatrix}
0 & 0 & 0 & 0 & a & c\\
0 & 0 & 0 & 0 & b & a\\
0 & 0 & 0 & 0 & 0 & 0\\
0 & 0 & 0 & 0 & 0 & 0\\
0 & 0 & 0 & 0 & 0 & 0\\
0 & 0 & 0 & 0 & 0 & 0
\end{bmatrix}
\right\}.
$$

We have
$$
[E_{25},A]=
\begin{bmatrix}
0 & 0 & 0 & 0 & -a & 0\\
0 & 0 & 0 & 0 & 0 & -a\\
0 & 0 & 0 & 0 & 0 & 0\\
0 & 0 & 0 & 0 & 0 & 0\\
0 & 0 & 0 & 0 & 0 & 0\\
0 & 0 & 0 & 0 & 0 & 0
\end{bmatrix}.
$$
Hence 
\begin{eqnarray*}
&&
\exp(t^{-1}\ad E_{25})(t\Lambda)
=
t\Lambda -(E_{15}+E_{26}),\\
&&
\exp(t^{-1}\ad E_{25})(\Lambda^3)
=
\Lambda^3,\\
&&
\exp(t^{-1}\ad E_{25})(\Lambda^5)
=
\Lambda^5.\\
\end{eqnarray*}
By Lemma \ref{lemma:ElementVsSpace}
$$
\lim_{t\to 0} \exp(t^{-1} \ad E_{25})(J)=
\left\{ 
\begin{bmatrix}
0 & 0 & 0 & b & -a & c\\
0 & 0 & 0 & 0 & -b & -a\\
0 & 0 & 0 & 0 & 0 & b\\
0 & 0 & 0 & 0 & 0 & 0\\
0 & 0 & 0 & 0 & 0 & 0\\
0 & 0 & 0 & 0 & 0 & 0
\end{bmatrix}
\right\}=:K.
$$

Let
$$
d_1(t):=\diag(1,t,1,1,t^{-1},1),\quad
d_2(t):=\diag(t,1,1,1,1,t^{-1}).
$$
Then by Lemma \ref{lemma:ElementVsSpace}
$$
\lim_{t\to 0} \Ad(d_i(t))(K)=\fka_{i}\qquad (i=1,2).
$$
Hence $\fka_{1}, \fka_{2}$
are contained in $\overline{\Ad(G)K}$.
\end{example}

\section{Basic deformations}

In this section,
 $\fka$ is  an abelian subalgebra of $\fkb$, and
we suppose that, 
\begin{equation}
\label{eqn:AbelianSubalgebra}
\text{
if $\alpha\in \Delta^+$ 
and $\fkg_\alpha
\subseteq \fka$, then
$\fkg_{\alpha+\beta}\subseteq \fka$
for all $\beta\in \Delta^+$.}
\end{equation}
By \eqref{eqn:DeltaA},
abelian ideals of $\fkb$ satisfy \eqref{eqn:AbelianSubalgebra}.

We prepare two deformations: unipotent deformation and
semisimple deformation, which are used many times in this paper.

\begin{lemma}
\label{UnipotentDeformation}
Let $\beta\in \Delta^+$.
\begin{enumerate}
\item
If $\alpha\in \Delta^+$ 
and $\fkg_\alpha\subseteq \fka$, then
$$
\fkg_\alpha\subseteq \lim_{t\to 0} \exp(t^{-1}\ad X_\beta)(\fka).
$$
\item
Let $\Gamma\in \fka$.
If $(\ad X_\beta)^i (\C\Gamma)\not\subseteq\fka$ and 
$(\ad X_\beta)^{j} (\C\Gamma)\subseteq \fka$ for all $j>i$, then
$$
(\ad X_\beta)^{i} (\C\Gamma)
\subseteq \lim_{t\to 0} \exp(t^{-1}\ad X_\beta)(\fka).
$$
\end{enumerate}
\end{lemma}

\begin{proof}
(1)\quad
By the assumption
\eqref{eqn:AbelianSubalgebra}, 
$(\ad X_\beta)^i(X_\alpha)\in \fka$ for all $i$.
Suppose that $k$ is the maximal $i$ with $(\ad X_\beta)^i(X_\alpha)\neq 0$.
Then 
$$
\exp(-t^{-1}\ad X_\beta)(X_\alpha)
=\sum_{i=0}^{k}\frac{1}{i!}(-t^{-1}\ad X_\beta)^i(X_\alpha)\in \fka
$$
for all $t\neq 0$.
Since
$$
X_\alpha=\exp(t^{-1}\ad X_\beta)\exp(-t^{-1}\ad X_\beta)(X_\alpha)
\in \exp(t^{-1}\ad X_\beta)(\fka)
$$
for all $t\neq 0$, we have by Lemma \ref{lemma:ElementVsSpace}
$$
X_\alpha \in \lim_{t\to 0} \exp(t^{-1}\ad X_\beta)(\fka).
$$

(2)\quad
Suppose that $k$ is the maximal $j$ with $(\ad X_\beta)^j(\Gamma)\neq 0$.
We inductively define Laurent polynomials $a_j(t)\in \C[t, t^{-1}]$ by
\begin{eqnarray*}
a_0(t)&=& a_1(t)=\cdots =a_i(t)=0\\
a_j(t)&=& \frac{1}{j!}t^{-j}-\sum_{q=i+1}^{j-1}\frac{1}{(j-q)!}t^{q-j}a_q(t) 
\qquad (i+1\leq j\leq k).
\end{eqnarray*}
Then
\begin{eqnarray*}
&&\exp(t^{-1}\ad X_\beta)(\sum_{q=i+1}^k a_q(t)(\ad X_\beta)^q(\Gamma))\\
&=&
\sum_{p,q}\frac{1}{p!}t^{-p}a_q(t)(\ad X_\beta)^{p+q}(\Gamma)\\
&=&
\sum_{j=i+1}^k\sum_{q=i+1}^j
\frac{1}{(j-q)!}t^{-(j-q)}a_q(t)(\ad X_\beta)^{j}(\Gamma)\\
&=&
\sum_{j=i+1}^k
\frac{1}{j!}t^{-j}(\ad X_\beta)^{j}(\Gamma).
\end{eqnarray*}
We have
\begin{eqnarray*}
&&\lim_{t\to 0}
\exp(t^{-1}\ad X_\beta)
(t^i(\Gamma-\sum_{q=i+1}^k a_q(t)(\ad X_\beta)^q(\Gamma)))\\
&=&
\lim_{t\to 0}(\frac{1}{i!}(\ad X_\beta)^i(\Gamma)+o(1))
=\frac{1}{i!}(\ad X_\beta)^i(\Gamma).
\end{eqnarray*}
Hence by Lemma \ref{lemma:ElementVsSpace}
$$
(\ad X_\beta)^{i} (\C\Gamma)\subseteq \lim_{t\to 0} \exp(t^{-1}\ad X_\beta)(\fka).
$$
\end{proof}

Let 
$T$ be the maximal torus of $G$ with Lie algebra $\fkh$.
Let $\chi_1,\ldots,\chi_n$ be the characters of $T$
corresponding to the simple roots
$\alpha_1,\ldots,\alpha_n$, respectively.
Let $\lambda_1,\ldots,\lambda_n$ be the $1$-parameter subgroups of $T$
such that $\chi_j(\lambda_i(t))=t^{\delta_{ij}}$.
For $\m=(m_1,\ldots,m_n)\in \Z^n$ and
 $\alpha=\sum_{j=1}^nd_j\alpha_j
\in \Delta^+$, set
\begin{equation}
\label{eqn:parenthesis}
(\m, \alpha)=\sum_{j=1}^n m_jd_j.
\end{equation}

Hence 
$$
\chi_\alpha(\prod_{j=1}^n \lambda_j^{m_j}(t))
=t^{(\m, \alpha)},
$$
where $\chi_\alpha$ is the character of $T$ corresponding to $\alpha$.

\begin{lemma}
\label{ToricDeformation}
Let $\Gamma=\sum_{\alpha\in \Delta^+}a_\alpha X_\alpha\in \fka$.
Suppose that $\fkg_\alpha\subseteq \fka$ if
$(\m, \alpha)<c$ and $a_\alpha\neq 0$.
Then
$$
\sum_{(\m, \alpha)=c}a_\alpha X_\alpha
\in 
\lim_{t\to 0}
\Ad(\prod_{j=1}^n \lambda_j^{m_j}(t))(\fka).
$$
\end{lemma}

\begin{proof}
We have
$$
\Ad(\prod_{j=1}^n \lambda_j^{m_j}(t))(\Gamma)
=
\sum_{\alpha\in \Delta^+}a_\alpha t^{(\m, \alpha)}X_\alpha.
$$
Hence
$$
\Ad(\prod_{j=1}^n \lambda_j^{m_j}(t))
(t^{-c}(\Gamma-\sum_{(\m,\alpha)<c}a_\alpha X_\alpha))
=
\sum_{(\m, \alpha)=c}a_\alpha X_\alpha + o(1).
$$
By Lemma \ref{lemma:ElementVsSpace}, we have
$$
\sum_{(\m, \alpha)=c}a_\alpha X_\alpha
\in 
\lim_{t\to 0}
\Ad(\prod_{j=1}^n \lambda_j^{m_j}(t))(\fka).
$$
\end{proof}

\section{Case $\fks\fkl(n+1,\C)$}

In this section, let $\fkg =\fks\fkl(n+1,\C)$,
and $\fkb$ the Lie subalgebra of upper triangular matrices.

Let $\fka$ be an $n$-dimensional abelian ideal of
$\fkb$.
Then by \eqref{eqn:DeltaA}
there exists a Young diagram
$\mu=(\mu_1\geq \mu_2\geq \cdots\geq \mu_l)$
with $|\mu|:=\sum_{i=1}^l\mu_i=n$ ($\mu\vdash n$)
such that
$$
\fka=\fka_\mu:=
\bigoplus_{k=1}^l\bigoplus_{j=1}^{\mu_k}
\C E_{j,\, n-k +2}.
$$

\begin{example}
Let $\mu=(\mu_1\geq \mu_2\geq \mu_3)=(4,4,1)$ and $n=9$.
Then the weight spaces of $\fka_\mu$ are the following places:
$$
\begin{tabular}{ccccccc}
\cline{2-7}
\multicolumn{1}{r}{1} & \multicolumn{1}{|c|}{}
& \multicolumn{1}{|c|}{}
& \multicolumn{1}{|c|}{}
& \multicolumn{1}{|c|}{$\bullet$}
& \multicolumn{1}{|c|}{$\bullet$}
& \multicolumn{1}{|c|}{$\bullet$}\\
\cline{2-7}
\multicolumn{1}{r}{2} & \multicolumn{1}{|c|}{}
& \multicolumn{1}{|c|}{}
& \multicolumn{1}{|c|}{}
& \multicolumn{1}{|c|}{}
& \multicolumn{1}{|c|}{$\bullet$}
& \multicolumn{1}{|c|}{$\bullet$}\\
\cline{2-7}
\multicolumn{1}{r}{3} & \multicolumn{1}{|c|}{}
& \multicolumn{1}{|c|}{}
& \multicolumn{1}{|c|}{}
& \multicolumn{1}{|c|}{}
& \multicolumn{1}{|c|}{$\bullet$}
& \multicolumn{1}{|c|}{$\bullet$}\\
\cline{2-7}
\multicolumn{1}{r}{4} & \multicolumn{1}{|c|}{}
& \multicolumn{1}{|c|}{}
& \multicolumn{1}{|c|}{}
& \multicolumn{1}{|c|}{}
& \multicolumn{1}{|c|}{$\bullet$}
& \multicolumn{1}{|c|}{$\bullet$}\\
\cline{2-7}
 & 5 & 6 & 7 & 8 & 9 & 10\\
\end{tabular}
$$
in the upper right block of size 
$\mu_1\times (n+1-\mu_1)=4\times 6$ of a square matrix 
of degree $n+1=10$.
\end{example}

Besides $\fka_\mu$,
we define another abelian Lie subalgebra $\fka_{\mu}'$
in the upper right block of size $\mu_1\times (n+1-\mu_1)$ 
by
$$
\fka_\mu' :=
\bigoplus_{k=1}^l\bigoplus_{j=1}^{\mu_k}
\C E_{\mu_1+1-j,\, \mu_1+\sum_{i> k}\mu_i +1}.
$$

\begin{example}
Let $\mu=(\mu_1\geq \mu_2\geq \mu_3)=(4,4,1)$ and $n=9$.
Then the weight spaces of $\fka_\mu'$ are the following places:
$$
\begin{tabular}{ccccccc}
\cline{2-7}
\multicolumn{1}{r}{1} & \multicolumn{1}{|c|}{}
& \multicolumn{1}{|c|}{$\bullet$}
& \multicolumn{1}{|c|}{}
& \multicolumn{1}{|c|}{}
& \multicolumn{1}{|c|}{}
& \multicolumn{1}{|c|}{$\bullet$}\\
\cline{2-7}
\multicolumn{1}{r}{2} & \multicolumn{1}{|c|}{}
& \multicolumn{1}{|c|}{$\bullet$}
& \multicolumn{1}{|c|}{}
& \multicolumn{1}{|c|}{}
& \multicolumn{1}{|c|}{}
& \multicolumn{1}{|c|}{$\bullet$}\\
\cline{2-7}
\multicolumn{1}{r}{3} & \multicolumn{1}{|c|}{}
& \multicolumn{1}{|c|}{$\bullet$}
& \multicolumn{1}{|c|}{}
& \multicolumn{1}{|c|}{}
& \multicolumn{1}{|c|}{}
& \multicolumn{1}{|c|}{$\bullet$}\\
\cline{2-7}
\multicolumn{1}{r}{4} & \multicolumn{1}{|c|}{$\bullet$}
& \multicolumn{1}{|c|}{$\bullet$}
& \multicolumn{1}{|c|}{}
& \multicolumn{1}{|c|}{}
& \multicolumn{1}{|c|}{}
& \multicolumn{1}{|c|}{$\bullet$}\\
\cline{2-7}
 & 5 & 6 & 7 & 8 & 9 & 10\\
\end{tabular}\quad .
$$
\end{example}

\begin{remark}
\label{rem:MuandMu'}
$\fka_\mu'$ and $\fka_\mu$ are conjugate to each other by 
$$
\begin{bmatrix}
P_\sigma & O\\
O & P_\tau
\end{bmatrix},
$$
where $P_\sigma$ and $P_\tau$ are respectively
the
permutation matrices corresponding to
\begin{eqnarray*}
\sigma\!&=&\!
\begin{pmatrix}
1 & 2 & \cdots & \mu_1 \\
\mu_1 & \mu_1-1 & \cdots & 1
\end{pmatrix} \quad \text{and}\\
\tau\!&=&\!
\begin{pmatrix}
\mu_1+1 &\! \cdots &\! \mu_1+\sum_{i>k}\mu_i +1 &\! \cdots &\! \sum_{i=1}^l \mu_i+1=n+1\\
n-l+2 & \cdots & n-k+2 & \cdots & n+1
\end{pmatrix}.
\end{eqnarray*}
Hence the statements $\fka_\mu\in \overline{\Ad(G)J}$ and
$\fka'_\mu\in \overline{\Ad(G)J}$ are equivalent.
\end{remark}

\begin{lemma}
\label{lem:i(h)}
For each $h=1,2,\ldots, n$, there exists a unique 
$i(h)$ such that $E_{i(h), i(h)+h}\in \fka'_\mu$. Explicitly,
\begin{equation}
\label{eqn:i(h)}
i(h)=\mu_1+1-(h-\sum_{i>k}\mu_i),
\end{equation}
or equivalently
\begin{equation}
\label{eqn:i(h)2}
i(h)+h=\mu_1+\sum_{i>k}\mu_i +1
\end{equation}
with $k$ satisfying $\sum_{i>k}\mu_i < h\leq \sum_{i\geq k}\mu_i$.
Note that $i(h)\leq \mu_1< i(h)+h$.
We have
$$
\fka_\mu' 
=\bigoplus_{h=1}^{n}\C E_{i(h),i(h)+h}.
$$
\end{lemma}

\begin{proof}
A weight of $\fka'_\mu$ 
corresponds to 
a place $(\mu_1+1-j, \mu_1+\sum_{i> k}\mu_i +1)$.
Then
its difference of components equals
$$
(\mu_1+\sum_{i> k}\mu_i +1)-(\mu_1+1-j)=\sum_{i> k}\mu_i +j.
$$
As $j$ runs over $[ 1, \mu_k]$, 
they are all different, and they cover $\{ 1,2,\ldots, n\}$.

When $h=\sum_{i> k}\mu_i +j$, we have
$$
i(h)=\mu_1+1-j=\mu_1+1-(h-\sum_{i> k}\mu_i).
$$
\end{proof}

For a vector $(z_1,z_2,\ldots,z_{n})\in \Q^{n}$ and $h,j =1,2,\ldots,n$
with $j+h\leq n+1$, put
$$
z_j(h):=\sum_{i=j}^{j+h-1}z_i.
$$

For $\mu\vdash n$, we consider the following system of inequalities:

\begin{equation}
\label{eqn:mu}
\tag{$IE_\mu$}
\left\{
\begin{array}{ll}
 z_{i(h)}(h)< z_j(h) & (1\leq h\leq n,\, j\leq n+1-h,\, j\neq i(h)),
\\
z_i> 0 & (1\leq i\leq n,\, i\neq \mu_1),
\\
z_{\mu_1}=0. &
\end{array}
\right.
\end{equation}

We give a proof of the following proposition in Appendix A:

\begin{proposition}
\label{prop:A:Inequality}
For any $\mu\vdash n$, there exists a solution of 
the system $(IE_\mu)$ in $\Z^{n}$.
\end{proposition}

Let $z=(z_1,\ldots, z_{n})\in \Z^{n}$ be a solution of 
the system $(IE_\mu)$.
Since  $(n+1)z=((n+1)z_1,\ldots, (n+1)z_{n})$ also satisfies 
 $(IE_\mu)$, we may assume $z_j\in (n+1)\Z$ for all $j$.

Define $w=(w_1,\ldots, w_{n+1})\in \Z^{n+1}$ by
\begin{equation}
\label{eqn:wj}
w_j:=\frac{\sum_{k=j}^{n}(n+1-k)z_k-\sum_{k=1}^{j-1}kz_k}{n+1} \qquad (j=1,\ldots, n+1).
\end{equation}
Then $\sum_{j=1}^{n+1}w_j=0$ and
\begin{equation}
\label{eqn:w-w=z}
w_j-w_{j+h}=\sum_{k=j}^{j+h-1}z_k=z_j(h).
\end{equation}

\begin{proposition}
\label{prop:Limit'}
Let $w\in \Z^{n+1}$ be the one defined in \eqref{eqn:wj}, and
let $t^w:=\diag( t^{w_1}, t^{w_2},\ldots , t^{w_{n+1}})\in SL(n+1, \C)$.
Then
$$
\lim_{t\to 0}\Ad(t^w)J=\fka_\mu'.
$$
\end{proposition}

\begin{proof}
Recall that $\Lambda:=\sum_{i=1}^{n}E_{i,i+1}$,
$\Lambda^h=\sum_{i=1}^{n+1-h}E_{i,i+h}$, and
$J=\bigoplus_{h=1}^{n}\C \Lambda^h$.
We have
$$
\lim_{t\to 0}\Ad(t^w)t^{-z_{i(h)}(h)}\Lambda^h
=\lim_{t\to 0}\sum_{j=1}^{n+1-h}t^{z_j(h)-z_{i(h)}(h)}E_{j,j+h}
=E_{i(h),i(h)+h}.
$$
Hence by Lemma \ref{lemma:ElementVsSpace}
$\lim_{t\to 0}\Ad(t^w)J=\fka_\mu'$.
\end{proof}

\begin{proof}[Proof of Theorem \ref{MainTheorem}]
Recall that we may take $K=J$
in the case of $\fks\fkl(n+1,\C)$.
For an $n$-dimensional abelian ideal $\fka$ of the Lie algebra
of upper triangular matrices in $\fks\fkl(n+1,\C)$,
there exists 
$\mu\vdash n$ such that $\fka=\fka_\mu$.
By Remark \ref{rem:MuandMu'} and Proposition
\ref{prop:Limit'},
$$
\fka_\mu
\in \overline{\Ad(G)J}.
$$
\end{proof}

\section{Main theorem for Types $B, C, D$}

Let $\fkg$ be a simple Lie algebra of Type B, C, or D.
Let $\alpha_1,\ldots,\alpha_n$ be the simple roots in $\Delta^+$;
we follow Bourbaki's notation \cite{Bourbaki}:
\begin{eqnarray*}
(B_n) & & \alpha_i=\varepsilon_i-\varepsilon_{i+1}\quad (i<n),\quad \alpha_n=\varepsilon_n,\\
&&
\Delta^+=\{ \varepsilon_i, \varepsilon_i-\varepsilon_j, \varepsilon_i+\varepsilon_j\,|\, i<j\};\\
(C_n) & & \alpha_i=\varepsilon_i-\varepsilon_{i+1}\quad (i<n),\quad \alpha_n=2\varepsilon_n,\\
&&
\Delta^+=\{ 2\varepsilon_i, \varepsilon_i-\varepsilon_j, \varepsilon_i+\varepsilon_j\,|\, i<j\};\\
(D_n) & & \alpha_i=\varepsilon_i-\varepsilon_{i+1}\quad (i<n),\quad 
\alpha_n=\varepsilon_{n-1}+\varepsilon_n,\\
&&
\Delta^+=\{ \varepsilon_i-\varepsilon_j, \varepsilon_i+\varepsilon_j\,|\, i<j\}.\end{eqnarray*}

Let $\fkb$ be the Borel subalgebra 
corresponding to $\Delta^+$, and
$\fka$ an $n$-dimensional abelian ideal of $\fkb$.
Recall that $\fka$ satisfies \eqref{eqn:DeltaA}.

\begin{lemma}
\label{lemma:RootsBCD}
The set $\Delta(\fka)$ consists of roots of form $\varepsilon_i+\varepsilon_j$
except 
\begin{enumerate}
\item $\Delta(\fka)=\{ \varepsilon_1, \varepsilon_1+\varepsilon_j\,|\, j\geq 2\}$ in Type $B_n$,
\item
$\Delta(\fka)=\{ \varepsilon_1-\varepsilon_n, \varepsilon_1+\varepsilon_j\,|\, j\geq 2\}$ in Type $D_n$,
\item
$\Delta(\fka)=\{ \varepsilon_2+\varepsilon_3, \varepsilon_1-\varepsilon_n, \varepsilon_1+\varepsilon_j\,|\, n>j\geq 2\}$ in Type $D_n$.
\end{enumerate}
\end{lemma}

\begin{proof}
Since the heights of the maximal roots of Types $B, C$, and $D$
are $2n-1, 2n-1$, and $2n-3$, respectively,
the heights of roots in $\Delta(\fka)$
are greater than or equal to
$n, n$, and $n-2$, respectively.
Thus we see the assertion for Types $B$ and $C$.

In Type $D$, the roots with height greater than $n-1$
are of form $\varepsilon_i+\varepsilon_j$, and
$\varepsilon_1-\varepsilon_n$ is the unique root not of form $\varepsilon_i+\varepsilon_j$
with height $n-1$.
Note that $\varepsilon_i-\varepsilon_j\in \Delta(\fka)$
leads to $\varepsilon_1-\varepsilon_n\in \Delta(\fka)$ by \eqref{eqn:DeltaA}.
Hence, if a root of form $\varepsilon_i-\varepsilon_j$
belongs to $\Delta(\fka)$, 
then it contains $\{ \varepsilon_1-\varepsilon_n, \varepsilon_1+\varepsilon_j\,|\, n>j\geq 2\}$.
If the remaining root of $\Delta(\fka)$ is not 
$\varepsilon_1+\varepsilon_n$ or
$\varepsilon_2+\varepsilon_3$, 
then $\Delta(\fka)$ must contain another root
since it is closed under the addition by the positive roots,
which contradicts $|\Delta(\fka)|=n$.
\end{proof}

We first consider the exceptional cases 
appearing in Lemma \ref{lemma:RootsBCD}.

\begin{proposition}
\label{prop:BCDexceptional}
Let $\Delta(\fka)$ be one of the following:
\begin{enumerate}
\item $\{ \varepsilon_1, \varepsilon_1+\varepsilon_j\,|\, j\geq 2\}$ in Type $B_n$,
\item
$\{ \varepsilon_1-\varepsilon_n, \varepsilon_1+\varepsilon_j\,|\, j\geq 2\}$ in Type $D_n$,
\item
$\{ \varepsilon_2+\varepsilon_3, \varepsilon_1-\varepsilon_n, \varepsilon_1+\varepsilon_j\,|\, n>j\geq 2\}$ in Type $D_n$ $(n\geq 5)$.
\item
Any in Type $D_4$.
\end{enumerate}
Then Theorem \ref{MainTheorem} holds, i.e.,
$$
\fka\in \overline{\Ad(G)K}.
$$
\end{proposition}

\begin{proof}
(1)\quad
The heights of $K$ are greater than or equal to $n$ (cf. Proposition \ref{prop:JtoK}).
For a root $\alpha$ of height $\geq n$, the coefficient of $\alpha_1$
is $1$ or $0$, and $1$ exactly when $\alpha=\varepsilon_1, \varepsilon_1+\varepsilon_j$
 $(j\geq 2)$. 
In other words, for a root $\alpha$ of height $\geq n$,
$$
\alpha(\lambda_1^{-1}(t))=
\left\{
\begin{array}{ll}
t^{-1} & (\alpha=\varepsilon_1, \varepsilon_1+\varepsilon_j \quad (j\geq 2))\\
1 & (\text{otherwise}).
\end{array}
\right.
$$
Hence by Lemma \ref{ToricDeformation},
we see
$$
\lim_{t\to 0}\Ad(\lambda_1^{-1}(t))(K)=\fka.
$$

(2)\quad
Let
$$
X_{\varepsilon_i-\varepsilon_j}:=
E_{i,j}-E_{2n+1-j,2n+1-i},\quad
X_{\varepsilon_i+\varepsilon_j}:=
E_{i,2n+1-j}-E_{j,2n+1-i}
$$
for $i<j$.
Similarly to the proof of (1), by Lemma \ref{ToricDeformation},
\begin{eqnarray*}
&&
\lim_{t\to 0}\Ad(\lambda_1^{-1}(t))(K)\\
&=&
\langle Z,
c_{\varepsilon_1-\varepsilon_n}X_{\varepsilon_1-\varepsilon_n}+
c_{\varepsilon_1+\varepsilon_n}X_{\varepsilon_1+\varepsilon_n},
X_{\varepsilon_1+\varepsilon_{n-1}},
\ldots,X_{\varepsilon_1+\varepsilon_2}\rangle\\
&=&\fka.
\end{eqnarray*}
Here $\langle A_1,\ldots,A_k \rangle$ means the
$\C$-vector space spanned by $A_1,\ldots, A_k$, and the last equation holds by the latter half 
of Proposition \ref{prop:JtoK}.

(3)\quad
Let $\beta:=\alpha_4+\cdots+\alpha_{n-2}+\alpha_n$ if $n\geq 6$,
and $\beta:=\alpha_5$ if $n=5$.
Then $\hite(\beta)=n-4$, and
$\varepsilon_1+\varepsilon_4=\beta+(\alpha_1+\cdots+\alpha_{n-2}+\alpha_{n-1})$.
Since $\gamma:=\alpha_1+\cdots+\alpha_{n-2}+\alpha_{n-1}$ is 
the unique root of height $n-1$ such that
$\beta+\gamma$ is a root,
$\ad(X_\beta)(\Lambda^{(n-1)})$ and
$\ad(X_\beta)(Z)$
are nonzero multiples of $X_{\varepsilon_1+\varepsilon_4}$.
Since no root of height $n$ or $n+1$ remains as a root after added by $\beta$,
we have
$\ad(X_\beta)(\Lambda^{(n)})=0$ and $\ad(X_\beta)(\Lambda^{(n+1)})=0$.
By $\hite(\beta)=n-4$,
$\ad(X_\beta)(\Lambda^{(n+j)})=0$ for $j\geq 2$.
By Lemma \ref{UnipotentDeformation}
\begin{eqnarray*}
\lim_{t\to 0}\exp (t^{-1}\ad X_\beta)(K)
&=&
\langle X_{\varepsilon_1+\varepsilon_4}, \Lambda^{(n-1)},\Lambda^{(n)},\ldots,
\Lambda^{(2n-3)}\rangle\\
&=&
\langle X_{\varepsilon_2+\varepsilon_3}, \Lambda^{(n-1)},\Lambda^{(n)},\ldots,
\Lambda^{(2n-3)}\rangle
=:\fka_1.
\end{eqnarray*}
Here the last equation holds since
$\Lambda^{(2n-5)}=c_{\varepsilon_1+\varepsilon_4}X_{\varepsilon_1+\varepsilon_4}+
c_{\varepsilon_2+\varepsilon_3}X_{\varepsilon_2+\varepsilon_3}$.
Again similarly to the proof of (1), by Lemma \ref{ToricDeformation},
\begin{eqnarray*}
&&\lim_{t\to 0}\Ad(\lambda_1^{-1}(t))(\fka_1)\\
&=&\langle X_{\varepsilon_2+\varepsilon_3}, 
c_{\varepsilon_1-\varepsilon_n}X_{\varepsilon_1-\varepsilon_n}+
c_{\varepsilon_1+\varepsilon_n}X_{\varepsilon_1+\varepsilon_n},
X_{\varepsilon_1+\varepsilon_{n-1}},\ldots,
X_{\varepsilon_1+\varepsilon_2}\rangle\\
&=:&\fka_2.
\end{eqnarray*}
Finally,
$$
\lim_{t\to 0}\Ad(\lambda_n(t))(\fka_2)
=\langle X_{\varepsilon_2+\varepsilon_3}, 
X_{\varepsilon_1-\varepsilon_n},
X_{\varepsilon_1+\varepsilon_{n-1}},\ldots,
X_{\varepsilon_1+\varepsilon_2}\rangle=\fka.
$$

(4)\quad
Let $\fkg$ be of type $D_4$.
The following is the list of non-simple positive roots:
$$
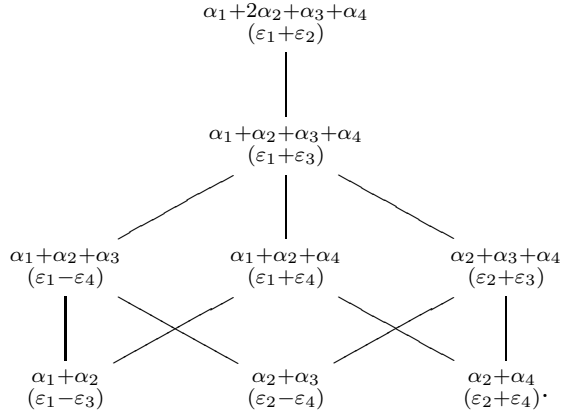
\xymatrix{
&
{\substack{\alpha_1+2\alpha_2+\alpha_3+\alpha_4\\ (\varepsilon_1+\varepsilon_2)}}
\ar@{-}[d]
&\\
&
{\substack{\alpha_1+\alpha_2+\alpha_3+\alpha_4\\ (\varepsilon_1+\varepsilon_3)}}
\ar@{-}[d]
\ar@{-}[ld]
\ar@{-}[rd]
&
\\
{\substack{\alpha_1+\alpha_2+\alpha_3\\ (\varepsilon_1-\varepsilon_4)}}
\ar@{-}[d]
\ar@{-}[rd]
&
{\substack{\alpha_1+\alpha_2+\alpha_4\\ (\varepsilon_1+\varepsilon_4)}}
\ar@{-}[ld]
\ar@{-}[rd]
&
{\substack{\alpha_2+\alpha_3+\alpha_4\\ (\varepsilon_2+\varepsilon_3)}}
\ar@{-}[d]
\ar@{-}[ld]
\\
{\substack{\alpha_1+\alpha_2\\ (\varepsilon_1-\varepsilon_3)}}
&
{\substack{\alpha_2+\alpha_3\\ (\varepsilon_2-\varepsilon_4)}}
&
{\substack{\alpha_2+\alpha_4\\ (\varepsilon_2+\varepsilon_4)}}.
}
$$
There exist the following three cases:
\begin{enumerate}
\item[(i)] $\Delta(\fka)=\{ \varepsilon_1-\varepsilon_4, \varepsilon_1+\varepsilon_j\,
|\, j=2,3,4\}$,
\item[(ii)] $\Delta(\fka')=\{ \varepsilon_2+\varepsilon_3, \varepsilon_1+\varepsilon_j\,
|\, j=2,3,4\}$,
\item[(iii)] $\Delta(\fka'')=\{ \varepsilon_2+\varepsilon_3, \varepsilon_1-\varepsilon_4,
\varepsilon_1+\varepsilon_j\,
|\, j=2,3\}$.
\end{enumerate}
The case (i) is included in (2).
Similarly to the case (i), we have
\begin{eqnarray*}
\lim_{t\to 0}\Ad(\lambda_4^{-1}(t))(K)
&=&
\fka',\\
\lim_{t\to 0}\Ad(\lambda_3^{-1}(t))(K)
&=&
\fka''.\\
\end{eqnarray*}
\end{proof}

In the rest of this section, we fix an abelian ideal $\fka$
of $\fkb$ such that
$\Delta(\fka)$
is none of the ones in Proposition \ref{prop:BCDexceptional}.
To prove $\fka\in \overline{\Ad(G)K}$ (Theorem \ref{MainTheorem}),
we define a sequence of abelian Lie subalgebras $K=\fka_1, \fka_2,
\ldots, \fka_{n+1}=\fka$ of $\fkb$
such that
$$
\fka_{l+1}\in \overline{\Ad(G)\fka_l} \qquad (l=1,2,\ldots, n).
$$

For $\beta=\varepsilon_i+\varepsilon_j\in \Delta(\fka)$ ($i\leq j$), 
put 
\begin{equation}
\label{eqn:i&j}
i=i(\beta)\qquad \text{and}\qquad j=j(\beta).
\end{equation}
Hence in particular
$$
i(\beta)\leq j(\beta).
$$

Note that for $\alpha\in \Delta(\fka)$
\begin{eqnarray}
\hite(\alpha)&=&2n+2-i(\alpha)-j(\alpha) \quad (B_n),
\nonumber\\
\hite(\alpha)&=&2n+1-i(\alpha)-j(\alpha) \quad (C_n),
\\
\hite(\alpha)&=&2n-i(\alpha)-j(\alpha) \quad (D_n),
\nonumber
\end{eqnarray}
and that for $\alpha, \beta\in \Delta(\fka)$
\begin{equation}
\label{lemma:TopandHeight}
\alpha\leq \beta
\Leftrightarrow
i(\alpha)\geq i(\beta),\quad
j(\alpha)\geq j(\beta).
\end{equation}
Here recall that $\alpha\leq \beta$ means
$\beta-\alpha\in \N\Delta^+$.

Set
$$
Y:=Y(\fka):=
\{ (i(\alpha), j(\alpha))\,|\, \alpha\in \Delta(\fka)\}.
$$
We sometimes identify $Y(\fka)$ with $\Delta(\fka)$.
Let $M$ be the set of $(i,j)\in Y$
with minimal $i$ among the elements in $Y$ with the same height
(or equivalently with the same $i+j$):
$$
M:=\{ (i,j)\in Y\,|\,
(i',j')\in Y,\, i+j=i'+j'
\Rightarrow i\leq i'\}.
$$
Put $L:=Y\setminus M$.


We introduce a total order $\prec$ into $\Delta(\fka)$ by 
\begin{equation}
\label{eqn:prec}
\alpha
\succ
\beta
\qquad
\Leftrightarrow
\qquad
\left\{
\begin{array}{c}
j(\alpha)<j(\beta)\\
\text{or}\\
j(\alpha)=j(\beta),\, i(\alpha)<i(\beta).
\end{array}
\right.
\end{equation}
Then $\alpha\geq \beta$ implies $\alpha\succeq\beta$,
and the maximal root is the biggest.

Enumerate the roots in $\Delta(\fka)$ according to $\prec$
from the biggest to the smallest,
starting with $1$. Let $\alpha(k)$ be the $k$-th root
in $\Delta(\fka)$.
(Hence $\alpha(1)$ is the maximal root $\gamma_0$.)
Note that there exist two cases for $\alpha(l-1)$ and $\alpha(l)$
(see Figure \ref{fig0}).

\begin{figure}
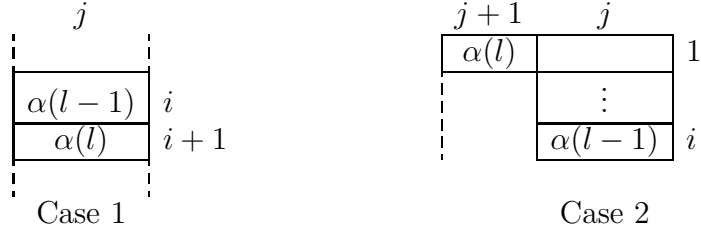

\hdashlinegap=2pt
\begin{tabular}{clcccl}
$j$ & \phantom{0} &  & $j+1$ & $j$ & \phantom{0}\\
\cline{4-5}
\multicolumn{1}{:c:}{\phantom{0}} &  &\phantom{1+2+3+4+5}
&\multicolumn{1}{|c}{$\alpha(l)$} &
\multicolumn{1}{|c|}{\phantom{0}} & $1$\\
\cline{1-1}\cline{4-5}
\multicolumn{1}{|c|}{$\alpha(l-1)$} & $i$
& \phantom{1+2+3+4+5} &
\multicolumn{1}{:c}{\phantom{0}} & \multicolumn{1}{|c|}{\vdots}
& \phantom{0}\\
\cline{1-1}\cline{5-5}
\multicolumn{1}{|c|}{$\alpha(l)$}& $i+1$ &
\phantom{1+2+3+4+5} &
\multicolumn{1}{:c}{\phantom{0}} &
\multicolumn{1}{|c|}{$\alpha(l-1)$} & $i$\\
\cline{1-1}\cline{5-5}
\multicolumn{1}{:c:}{\phantom{0}}& \phantom{0} &
\phantom{1+2+3+4+5} &
\phantom{0} &
\phantom{0} & \phantom{0}\\
\text{Case 1}& &&&
\text{Case 2}&
 \end{tabular} 
\caption{$\alpha(l-1)$ and $\alpha(l)$}
\label{fig0}
\end{figure}

For $l=1,2,\ldots, n+1$, put
$$
{Y}(l):=\{ \alpha(1), \alpha(2),\ldots, \alpha(l-1)\},\,\,
L(l):={Y}(l)\cap L,\,\,
M(l):={Y}(l)\cap M.
$$
We divide $M$ into two sets $M_1$ and $M_2$:
$$
M_1:=\{ \alpha\in M\,|\,  i(\alpha)=1 \},\quad M_2:=M\setminus M_1.
$$

\begin{definition}
A root $\alpha\in \Delta(\fka)$ is called a {\it source}
if there exist no $\beta\in \Delta(\fka)$ and $\gamma\in
\Delta^+$ such that $\alpha=\beta+\gamma$, i.e., a source is
a minimal element of $\Delta(\fka)$ with respect to $\leq$.
\end{definition}

By \eqref{lemma:TopandHeight} the following is obvious:

\begin{proposition}
\label{prop:source}
The set of sources equals
$$
\{ \varepsilon_i+\varepsilon_j\in \Delta(\fka)\,|\,
\varepsilon_{i+1}+\varepsilon_j, \varepsilon_i+\varepsilon_{j+1}\notin
\Delta(\fka)\}.
$$
\end{proposition}

\begin{example}
\label{ex:B7}
Let $\fkg$ be of type $B_7$, and let
$$
\Delta(\fka):=
\{
\varepsilon_1+\varepsilon_2,
\varepsilon_1+\varepsilon_3,
\varepsilon_2+\varepsilon_3,
\varepsilon_1+\varepsilon_4,
\varepsilon_2+\varepsilon_4,
\varepsilon_3+\varepsilon_4,
\varepsilon_1+\varepsilon_5
\}.
$$
The sources are $\varepsilon_3+\varepsilon_4,
\varepsilon_1+\varepsilon_5$.

By the definition of $\prec$,
$$
\alpha(1)=\varepsilon_1+\varepsilon_2,
\alpha(2)=\varepsilon_1+\varepsilon_3,
\ldots,
\alpha(6)=\varepsilon_3+\varepsilon_4,
\alpha(7)=\varepsilon_1+\varepsilon_5;
$$
\begin{figure}[!ht]
\begin{tiny}
\begin{center}
\begin{tabular}{ccccl}
5 & 4 & 3 & 2 & $j$/$i$\\
\cline{1-4}
 \multicolumn{1}{|c}{7}
&
 \multicolumn{1}{|c|}{4}
&
 \multicolumn{1}{c|}{2}
&
 \multicolumn{1}{c|}{1}
&
1\\
\cline{1-4}
\phantom{0}
&
 \multicolumn{1}{|c|}{5}
&
 \multicolumn{1}{c|}{3}
&
\phantom{0}
&
2\\
\cline{2-3}
\phantom{0}
&
 \multicolumn{1}{|c|}{6}
&
\phantom{0}
&
\phantom{0}
&
3\\
\cline{2-2}
 \end{tabular}.
\end{center}
\end{tiny}
\label{fig1}
\end{figure}

\noindent
Here we put $k$ in the box at $(i,j)$ when $\alpha(k)=\varepsilon_i+\varepsilon_j$ $(i<j)$.
We have
\begin{eqnarray*}
M &=&
\{ \alpha(1), \alpha(2), \alpha(4), \alpha(6), \alpha(7)\}\\
L &=&
\{ \alpha(3), \alpha(5)\},
\end{eqnarray*}
and
$$
M_1=\{ \alpha(1), \alpha(2), \alpha(4), \alpha(7)\},\quad
M_2=\{ \alpha(6)\}.
$$
\end{example}

\bigskip
For $\alpha\in \Delta(\fka)$, let $s(\alpha)\in \Delta(\fka)$
denote the biggest (with respect to $\prec$) source $\beta$ 
with $\beta\leq \alpha$.
By \eqref{lemma:TopandHeight},
\begin{equation}
\label{eqn:s(alpha)&alpha}
i(s(\alpha))\geq i(\alpha),\quad j(s(\alpha))\geq j(\alpha).
\end{equation}

\begin{lemma}
\label{lemma:sourceM2}
For $\alpha\in M_2$,
$j(\alpha)=j(s(\alpha))$.
\end{lemma}

\begin{proof}
This is clear since $\alpha\in M_2$ implies
$(i(\alpha), j(\alpha)+1)\notin Y$.
\end{proof}

\begin{lemma}
\label{lemma:sourceOfX}
Let $\alpha, \beta\in \Delta(\fka)$
satisfy
$j(\alpha)=j(\beta)$.

Then
$s(\alpha)=s(\beta)$.
\end{lemma}

\begin{proof}
We may suppose that $i(\alpha)< i(\beta)$.
Then $s(\beta)\leq \beta \leq \alpha$, and hence
$s(\beta)\preceq s(\alpha)\leq \alpha$.

If $s(\alpha)\leq \beta$, then we have $s(\alpha)=s(\beta)$
by the definition of $s(\beta)$.

Suppose that $s(\alpha)\not\leq \beta$.
Then $i(s(\alpha))<i(\beta) (\leq i(s(\beta)))$,
because $j(s(\alpha))\geq j(\alpha)=j(\beta)$.
By $s(\beta)\preceq s(\alpha)$, we have
$j(s(\beta))\geq j(s(\alpha))$.
Hence we have $s(\beta)\lneq s(\alpha)$.
This implies that $s(\alpha)$ is not a source.
\end{proof}

For $l\geq 1$, we define $1\leq t_l\leq \infty$, which plays
an important role in the inductive proof of Theorem \ref{thm:Main}.

\begin{equation}
\label{eqn:t_j}
t_l:=\left\{
\begin{array}{l}
\infty \qquad\quad (l=1)\\
\infty \qquad\quad 
(Y\ni\exists 2\varepsilon_j\leq \alpha(l-1)\notin M_2 \qquad \text{(Type $C$)})\\
\infty \qquad\quad (Y\ni\exists \varepsilon_{j-1}+\varepsilon_j\leq \alpha(l-1)\notin M_2 \quad 
\text{(Types $B,D$)})\\
\min\{ i(s(\beta))\,|\, \beta\succeq \alpha(l-1)\}
\qquad\qquad (\text{otherwise}).
\end{array}
\right.
\end{equation}

\begin{example}
In Example \ref{ex:B7},
$t_1=t_2=\cdots=t_6=\infty, t_7=3, t_8=1$.
\end{example}

\begin{lemma}
\label{lem:t_l's}
We have
$$
t_1\geq t_2\geq \cdots,
$$
and, if $t_l\neq \infty$, then
$t_l=i(s(\alpha(l-1)))$.
Moreover $t_{l+1}\neq t_l$ implies 
$t_l=\infty$ and $\alpha(l)\in M_2$, or
$\alpha(l)\in M_1$.
\end{lemma}

\begin{proof}
Recall that there exist two cases for $\alpha(l-1)$ and $\alpha(l)$
(see Figure \ref{fig0}).
First we prove that $t_{l+1}=\infty$ implies $t_{l}=\infty$.
This is clear for Case 1.
(Note that $\alpha(l-1)\in M_2$ implies $\alpha(l)\in M_2$, or
$\alpha(l)\notin M_2$ implies $\alpha(l-1)\notin M_2$.)
In Case 2, $t_{l+1}=\infty$ implies that
$2\varepsilon_{j(\alpha(l))}\in Y$ (Type $C$)
and $\varepsilon_{j(\alpha(l))-1}+\varepsilon_{j(\alpha(l))}\in Y$
(Types $B,D$), respectively.
Then $\alpha(l-1)=2\varepsilon_{j(\alpha(l))-1}$ 
(Type $C$)
and $\alpha(l-1)=\varepsilon_{j(\alpha(l))-2}+\varepsilon_{j(\alpha(l))-1}$
(Types $B,D$), respectively.
Hence
we have $t_l=\infty$.

Suppose that $t_{l+1}\neq \infty$. Then
by $\{ i(s(\beta))\,|\, \beta \succeq \alpha(l)\}
\supseteq \{ i(s(\beta))\,|\, \beta \succeq \alpha(l-1)\}$,
clearly $t_{l+1}\leq t_{l}$.


Next we show $i(s(\alpha(l))\leq i(s(\alpha(l-1))$ for any $l$.
If $s(\alpha(l))= s(\alpha(l-1))$, then this is clear.
Suppose that $s(\alpha(l))\neq s(\alpha(l-1))$.
Then $\alpha(l-1)$ and $\alpha(l)$ are in Case 2 (Figure \ref{fig0}).
Thus $j(\alpha(l))= j(\alpha(l-1))+1$ and $\alpha(l)\in M_1$.
If $i(s(\alpha(l))>i(\alpha(l-1))$, then
$$
i(\alpha(l-1))=
\left\{
\begin{array}{ll}
j(\alpha(l-1))-1 & (\text{Types $B,D$})\\
j(\alpha(l-1)) & (\text{Type $C$}),
\end{array}
\right.
$$
because otherwise
$s(\alpha(l))\leq (i(\alpha(l-1))+1, j(\alpha(l-1)))$
by the inequalities
$j(s(\alpha(l)))\geq j(\alpha(l))= j(\alpha(l-1))+1>j(\alpha(l-1))$ and
\eqref{lemma:TopandHeight}.
But this contradicts \eqref{eqn:DeltaA},
since  $(i(\alpha(l-1))+1, j(\alpha(l-1)))$ does not belong to $Y$.

Thus $s(\alpha(l-1))\neq s(\alpha(l))$ implies
$i(s(\alpha(l))\leq i(\alpha(l-1))\leq i(s(\alpha(l-1))$.
Here the last inequality holds by \eqref{lemma:TopandHeight}.

Hence in any case
$i(s(\alpha(l))\leq i(s(\alpha(l-1))$, and thus
$t_l=i(s(\alpha(l-1)))$ if $t_l\neq\infty$.

Finally suppose that $t_{l+1}\neq t_l$ and $\alpha(l)\notin M_1$.
Then $\alpha(l-1)$ and $\alpha(l)$ are in Case 1 (Figure \ref{fig0}).
By Lemma \ref{lemma:sourceOfX}, $s(\alpha(l-1))= s(\alpha(l))$.
Hence $t_l=\infty, t_{l+1}=i(s(\alpha(l))), \alpha(l-1)\notin M_2$,
and $\alpha(l)\in M_2$.
\end{proof}

\begin{lemma}
\label{lem:t_j}
$t_l\geq i(\alpha(l))$.
\end{lemma}

\begin{proof}
We may suppose that $t_l\neq \infty$.
By Lemma \ref{lem:t_l's},
$t_l=i(s(\alpha(l-1)))$.
If $\alpha(l-1)$ and $\alpha(l)$ are in Case 2 (Figure \ref{fig0}),
then $i(\alpha(l))=1$, and the assertion is clear.
If $\alpha(l-1)$ and $\alpha(l)$ are in Case 1 (Figure \ref{fig0}),
then $s(\alpha(l))= s(\alpha(l-1))$ by Lemma \ref{lemma:sourceOfX},
and hence by \eqref{lemma:TopandHeight}
$$
t_l=i(s(\alpha(l-1)))=i(s(\alpha(l)))\geq i(\alpha(l)).
$$
\end{proof}

\begin{theorem}
\label{thm:Main}
Let $\fkg$ be a simple Lie algebra of type $B, C$, or $D$.
Then
we have
Theorem \ref{MainTheorem}, i.e.,
$\fka\in
\overline{\Ad(G)K}$.
\end{theorem}

\begin{proof}
We already proved the assertion in four cases
(Proposition \ref{prop:BCDexceptional}).
We suppose that $\fka$ is none of those cases.

Set $\hite(Y):=\{ \hite(\alpha)\,|\, \alpha\in Y\}$.
Then, by the definition of $M$, 
for each $h\in \hite(Y)$, there exists a unique $\alpha\in M$ with
$\hite(\alpha)=h$.

For $k$ with $\hite(\gamma_0)-n+k< \min \hite({Y})$,
Put
\begin{equation*}
\Theta_k:=
\left\{
\begin{array}{ll}
Z & \text{if $k=1$ in type $D_n$,}\\
\Lambda^{(\hite(\gamma_0)-n+k)}
& \text{otherwise.}
\end{array}
\right.
\end{equation*}
Here recall Example \ref{D-1} for $Z$ and
Proposition \ref{prop:JtoK} for $\Lambda^{(k)}$.

Then
$$
K=\bigoplus_{\hite(\gamma_0)-n+k<\min (\hite(Y))}
\C \Theta_k \bigoplus
\bigoplus_{\alpha\in M}\C \Lambda^{(\hite(\alpha))}.
$$

For $k$ and
$\Gamma=\sum_{\alpha\in \Delta^+} a_\alpha X_\alpha$,
put
\begin{equation}
\label{eqn:Pgamma}
P_{\leq k}(\Gamma):=
\sum_{i(\alpha)\leq k} a_\alpha X_\alpha.
\end{equation}

Set
\begin{eqnarray*}
\fka_l&:=&
\bigoplus_{\alpha\in {Y}(l)}
\C X_\alpha
\bigoplus
\bigoplus_{\overset{\alpha\in M_2\setminus {Y}(l)}{
s(\alpha)=s(\beta)\,\,
(\exists \beta\in M_2(l))}}
\C X_{\alpha}\\
&&\quad\bigoplus \!\!\!\!
\bigoplus_{\overset{k>\sharp L(l)}{
\hite(\gamma_0)-n+k< \min \hite({Y})}}
\C P_{\leq t_{l}}(\Theta_k)
\bigoplus
\bigoplus_{\alpha\in M_1\setminus
{Y}(l)}
\C P_{\leq t_{l}}(\Lambda^{(\hite(\alpha))})\\
&&\quad
\bigoplus
\bigoplus_{\underset{s(\alpha)\neq s(\beta)\,
\, (\forall \beta\in M_2(l))}{\alpha\in M_2\setminus {Y}(l)}}
\C P_{\leq t_{l}}(\Lambda^{(\hite(\alpha))}).
\end{eqnarray*}

Then $\fka_1=K$, and $\fka_{n+1}=\fka$.
Note that $\fka_l$ satisfies the assumption \eqref{eqn:AbelianSubalgebra};
suppose that $\alpha\in M_2\setminus Y(l)$ and $s(\alpha)=s(\beta)$ with $\beta\in M_2(l)$,
and that $\gamma, \alpha+\gamma\in \Delta^+ $.
If $\alpha+\gamma\notin Y(l)$,
then $\beta\succeq \alpha(l-1)\succ \alpha+\gamma\succ \alpha$.
Hence by the definition of $\succ$ and Lemma \ref{lemma:sourceM2}
$$
j(\alpha)\geq j(\alpha+\gamma)\geq j(\beta)=j(s(\beta))=j(s(\alpha))\geq j(\alpha).
$$
Hence $j(\alpha+\gamma)=j(\beta)$, and by Lemma \ref{lemma:sourceOfX}
$s(\alpha+\gamma)=s(\beta)$.
Moreover $\beta\succ \alpha+\gamma$, $j(\alpha+\gamma)=j(\beta)$,
and $\beta\in M_2$ imply $\alpha+\gamma\in M_2$.
We have thus checked the assumption \eqref{eqn:AbelianSubalgebra}
for $\fka_l$.

We show
\begin{equation}
\label{eqn:fkal}
\fka_{l+1}\in
\overline{\Ad(G)\fka_l}\qquad (l=1,2,\ldots, n).
\end{equation}
Then, inductively, we have
Theorem \ref{MainTheorem}.

The proof of \eqref{eqn:fkal} is divided into three cases according to
$\alpha(l)\in M_1, L$, or $M_2$.

\medskip
(Case 1: $\alpha(l)\in M_1$.)
In this case, $\alpha(l-1)$ and $\alpha(l)$ are in Case 2 (Figure \ref{fig0}).
Since each root appearing in $P_{\leq t_l}(\Lambda^{\hite(\alpha(l))})$
except $\alpha(l)$ belongs to $Y(l)$,
the root vector $X_{\alpha(l)}$ belongs to $\fka_l$.
If $t_{l+1}=t_l$, then $\fka_{l+1}=\fka_l$.

Next suppose that $t_{l+1}< t_l$. 
We prove that
\begin{equation}
\label{MainTheorem:M1}
\fka_{l+1}=
\lim_{t\to 0}\Ad(\lambda_{t_{l+1}}^{-1}(t))\fka_l.
\end{equation}
First we show $t_{l+1}\leq i(\alpha(l-1))$.
If $t_{l+1}<t_l<\infty$,
then
$\alpha(l-1)$ is a source, and $t_l=i(\alpha(l-1))$ by Lemma \ref{lem:t_l's}.
If $t_l=\infty$ and $t_{l+1}<\infty$, then
$(i(\alpha(l-1))+1,j(\alpha(l-1))+1)$ does not belong to $Y$.
Since $j(\alpha(l))=j(\alpha(l-1))+1$,
this implies $t_{l+1}=i(s(\alpha(l)))\leq i(\alpha(l-1))$.
Hence we have proved $t_{l+1}\leq i(\alpha(l-1))$.

If the coefficient of $\alpha_{t_{l+1}}$ in a root $\alpha$
is $2$, then $\alpha$ is of form $\varepsilon_i+\varepsilon_j$ with
$i\leq j\leq t_{l+1}$.
Since $i\leq j\leq t_{l+1}\leq i(\alpha(l-1))\leq j(\alpha(l-1))$,
we have $\alpha(l-1)\leq \alpha$, and hence $X_\alpha\in \fka_l$.
Hence the linear combinations with roots whose coefficients of
$\alpha_{t_{l+1}}$ are $1$ survive under $\lim_{t\to 0}\Ad(\lambda_{t_{l+1}}^{-1}(t))$
by Lemma \ref{ToricDeformation};
$$
P_{\leq t_{l+1}}(\Lambda^{(h)}),  P_{\leq t_{l+1}}(\Theta_k)\in
\lim_{t\to 0}\Ad(\lambda_{t_{l+1}}^{-1}(t))\fka_l.
$$
Hence we have proved \eqref{MainTheorem:M1}.

\medskip
(Case 2: $\alpha(l)\in L$.)
In this case, $\alpha(l-1)$ and $\alpha(l)$ are in Case 1 (Figure \ref{fig0}),
and $t_{l+1}=t_l$ by Lemma \ref{lem:t_l's}.
Suppose that $l=3$ and $\fkg$ is of type $D_n$.
Then $\fka_2=\fka_3$ and
$\alpha(3)=\varepsilon_2+\varepsilon_3$.
Let $\beta:=\alpha_4+\cdots+\alpha_{n-2}+\alpha_n$ if $n\geq 6$,
and $\beta:=\alpha_5$ if $n=5$.
As in the proof of Proposition \ref{prop:BCDexceptional}
(3),
$\C [X_\beta, Z]=\C X_{\varepsilon_1+\varepsilon_4}$.
Hence
$
\langle  [X_\beta, Z], \Lambda^{(2n-5)}\rangle
= \langle X_{\varepsilon_1+\varepsilon_4}, 
X_{\varepsilon_2+\varepsilon_3}\rangle
$
and
$$
\lim_{t\to 0}\exp (t^{-1}\ad X_\beta)(\fka_3)
=\fka_4.
$$

Suppose that $l\neq 3$ or $\fkg$ is not of type $D_n$.
Let $h:=\hite(\Theta_{\sharp L(l)+1})$.
Recall that $\Theta_{\sharp L(l)+1}=
\Lambda^{(h)}$
and $h=\hite(\gamma_0)-n+\sharp L(l)+1$.

Put $i:=i(\alpha(l))$ and $j:=j(\alpha(l))$.
By Lemma \ref{lem:t_l's}
$t_{l+1}=t_l$, and note that
 $i\geq 2$, since $\alpha(l)\in L$.
We express $\alpha(l)$ as a sum of simple roots
in the following order: 

\begin{enumerate}
\item[($B_n$)]
\begin{eqnarray*}
\alpha(l)&=&
\sum_{i\leq k<j}\alpha_k+2\sum_{j\leq k \leq n}\alpha_k\\
&=&
\alpha_i+\cdots +\alpha_j +\cdots +\alpha_n+
\alpha_n +\cdots +\alpha_j,
\end{eqnarray*}
\item[($C_n$)]
\begin{eqnarray*}
\alpha(l)&=&
\sum_{i\leq k<j}\alpha_k+2\sum_{j\leq k < n}\alpha_k+\alpha_n\\
&=&
\alpha_i+\cdots +\alpha_j +\cdots +
\alpha_n +\cdots +\alpha_j,
\end{eqnarray*}
\item[($D_n$)]
\begin{eqnarray*}
\alpha(l)&=&
\sum_{i\leq k<j}\alpha_k+2\sum_{j\leq k \leq n-2}\alpha_k
+\alpha_{n-1}+\alpha_n\\
&=&
\alpha_i+\cdots +\alpha_j +\cdots +\alpha_{n-2}+\alpha_n+
\alpha_{n-1} +\cdots +\alpha_j,
\end{eqnarray*}
\end{enumerate}

In the above, let
$\gamma$ be the sum of the first $h$ simple roots,
and $\beta$ the rest.
Note that $\beta$ and $\gamma$ can be defined
since we have $h<\min \hite(Y)\leq \hite(\alpha(l))$,
and that they are in fact roots
since $h=n+\sharp L(l)$ for $B_n, C_n$ ($n-2+\sharp L(l)$ for $D_n$ respectively).

We prove that
\begin{equation}
\label{eqn:LtoLplus1}
\fka_{l+1}=\lim_{t\to 0}\exp t^{-1}\ad X_\beta (\fka_l).
\end{equation}

First we show
\begin{equation}
\label{Main:eqn:1st}
[ X_\beta, \C P_{\leq t_l}(\Theta_{\sharp L(l)+1})]
=\C X_{\alpha(l)}.
\end{equation}

By Lemma \ref{lem:t_j},
$t_l\geq i(\alpha(l))=i$,
Hence $\gamma$ certainly
appears in $P_{\leq t_l}(\Theta_{\sharp L(l)+1})$.
If $\gamma'\neq\gamma$ is a root of height $h$, and if
$\gamma'+\beta$ is also a root, then $\gamma'$ should be 
of the form
$\alpha_{j-1}+\alpha_{j-2}+\cdots+ \alpha_{j-h}$.
In Case $B$ or $C$,
no such root appears in $P_{\leq t_l}(\Theta_{\sharp L(l)+1})$,
since $h\geq n$.
In Case $D$, we have $h \geq n-2$.
If $h\geq n-1$, then
$\gamma'$ cannot exist.
If $h= n-2$, then $\sharp L(l)=0$, and
hence $\alpha(l)$ is the first root in $L$;
$\alpha(l)=\varepsilon_2+\varepsilon_3$.
If $\gamma'$ exists, then $1\leq j-h=3-(n-2)$.
Namely $n\leq 4$.
Hence the only possible case is the one
when
$n=4, \gamma=\alpha_2+\alpha_4, \beta=\alpha_3$, and
$\gamma'=\alpha_2+\alpha_1$,
which we excluded in the beginning.
Hence we have proved \eqref{Main:eqn:1st}.
Since $X_\alpha\in \fka_l$ for all $\alpha>\alpha(l)$,
by Lemma \ref{UnipotentDeformation}
$$
X_{\alpha(l)}\in \lim_{t\to 0}\exp t^{-1}\ad X_\beta (\fka_l).
$$

As in the previous paragraph, for $k\geq 1$,
$$
[X_\beta, P_{\leq t_l}
(\Lambda^{(h+k)})]\in
\C X_\alpha,
$$
where 
$\alpha=
\alpha(l)+\alpha_{i-1}+\alpha_{i-2}+\cdots+\alpha_{i-k}$.
Hence, even if $\alpha$ is a root, we have
\begin{equation}
\label{Main:eqn:2nd}
[X_\beta, P_{\leq t_l}
(\Lambda^{(h+k)})]\in
\fka_l\qquad (k\geq 1)
\end{equation}
since $\alpha> \alpha(l)$.
Hence by Lemma \ref{UnipotentDeformation}
\begin{equation}
P_{\leq t_{l+1}}
(\Lambda^{(h+k)})=
P_{\leq t_l}
(\Lambda^{(h+k)})
\in \lim_{t\to 0}\exp t^{-1}\ad X_\beta (\fka_l).
\end{equation}

Since ${Y}(l)$ is closed under
addition of a positive root,
it is clear that
\begin{equation*}
[ X_\beta, X_\alpha]
\in \fka_l
\qquad \text{for $\alpha\in {Y}(l)$}.
\end{equation*}
Hence by Lemma \ref{UnipotentDeformation}
\begin{equation}
\label{Main:eqn:3rd}
X_\alpha\in
\lim_{t\to 0}\exp t^{-1}\ad X_\beta (\fka_l)
\qquad \text{for $\alpha\in {Y}(l)$}.
\end{equation}

Finally, we prove
\begin{equation}
\label{Main:eqn:4th}
[ X_\beta, X_\alpha]
\in \fka_l
\end{equation}
for $\alpha\in M_2\setminus{Y}(l)$
with $s(\alpha)=s(\beta')$ and
$\beta'\in M_2(l)$.
Note that in this case 
$M_2\setminus Y(l)=M_2\setminus Y(l+1)$
and
$M_2(l)=M_2(l+1)$.
Since $\hite(\alpha)\geq \min \hite(Y)>\hite(\Theta_{\sharp L(l)+1})=h$,
by the similar argument to \eqref{Main:eqn:2nd},
we see
$[ X_\beta, X_\alpha]\in \C X_{\alpha'}$ with $\alpha'> \alpha(l)$.
We have thus proved \eqref{Main:eqn:4th},
and hence by Lemma \ref{UnipotentDeformation}
$$
X_\alpha\in
\lim_{t\to 0}\exp t^{-1}\ad X_\beta (\fka_l)
$$
for $\alpha\in M_2\setminus{Y}(l+1)$
with $s(\alpha)=s(\beta')$ and
$\beta'\in M_2(l+1)$.
Hence we have proved
\eqref{eqn:LtoLplus1}.

\medskip
(Case 3: $\alpha(l)=\varepsilon_{i_l}+\varepsilon_{j_l}\in M_2$.
($i_l<j_l$ for $B,D$; $i_l\leq j_l$ for $C$.))

In this case, $\alpha(l-1)$ and $\alpha(l)$ are in Case 1 (Figure \ref{fig0}),
and $s(\alpha(l-1))=s(\alpha(l))$ (Lemma \ref{lemma:sourceOfX}).
By Lemma \ref{lem:t_l's}, there exist two cases:
\begin{enumerate}
\item[(a)] $t_{l+1}=t_l< \infty$,
\item[(b)] $t_{l+1}<\infty, t_l=\infty, \alpha(l-1)\notin M_2$.
\end{enumerate}
Since $\alpha(l)=\varepsilon_{i_l}+\varepsilon_{j_l}\in M_2$
implies that there exist no $\alpha=\varepsilon_{i_l}+\varepsilon_{j}\in
Y$ for $j>j_l$, we see $j(s(\alpha(l)))=j_l$.

First consider Case (a); suppose that $t_{l+1}=t_l< \infty$.
If there exist $\beta\in M_2$ such that $\alpha(l)\prec \beta$ and
$s(\beta)=s(\alpha(l))$, then
$\fka_{l+1}=\fka_l$.

Suppose that there exist no such $\beta\in M_2$.
Then $s(\alpha(l))=\varepsilon_{t_{l+1}}+\varepsilon_{j_l}$
by Lemma \ref{lem:t_l's}.
Put
$$
\m=-a\e_{t_{l+1}}-b\e_{j_l}\in \Z^n,
$$
where $a>b>0$.
Then
$(\m, \varepsilon_i+\varepsilon_j)$ (cf. \eqref{eqn:parenthesis}) are as follows:

\begin{center}
\begin{tabular}{|c|c|c|c}
\label{tab1}
  $\cdots >j_l  $ & $j_l\geq \cdots $ & $t_{l+1}\geq\cdots $ & $j/i$ \\
\hline
 & & &$\vdots$ \\
  $-a-b$ & $-a-2b$ & $-2a-2b$ & $t_{l+1}$\\
\hline
 & & & $\vdots$\\
  $-b$ & $-2b$ & $-a-2b$ & $j_l$\\
\hline
  $0$ & $-b$ & $-a-b$ & \vdots\\
\hline
\end{tabular}.
\end{center}

We prove that 
\begin{equation}
\label{eq:Case3}
\fka_{l+1}=\lim_{t\to 0}\Ad(\lambda^\m(t))(\fka_l).
\end{equation}

Let $\alpha\in M_2$ satisfy $s(\alpha)=s(\alpha(l))$.
Then
$j(\alpha)=j(\alpha(l))=j_l$, and
$i_l=i(\alpha(l))\leq i(\alpha)\leq i(s(\alpha(l)))=t_{l+1}$,
since $\alpha\preceq \alpha(l)$.
Hence
$\alpha$ is the unique root in $P_{\leq t_l}(\Lambda^{(\hite(\alpha))})$
with $\m$-weight $-a-2b$ ($-2a-2b$ if $t_{l+1}=j_l$ in Type $C$)
outside of ${Y}(l)$.
Hence by Lemma \ref{ToricDeformation}
$$
X_\alpha\in \lim_{t\to 0}\Ad(\lambda^\m(t))(\C 
P_{\leq t_l}(\Lambda^{(\hite(\alpha))})).
$$

Let $\alpha\in M\setminus {Y}(l)$ and
$s(\alpha)\neq s(\alpha(l))$.
Then $\alpha\prec \alpha(l)$, and
$j(\alpha)>j(\alpha(l))$ by Lemma \ref{lemma:sourceOfX}.
Since $\alpha(l)\in M_2$, we have $i(\alpha)<i(\alpha(l))$.
We show 
\begin{equation}
\label{eqn:h-hl}
\hite(\alpha)>\hite(\alpha(l)).
\end{equation}
Suppose otherwise.
Then $j(\alpha)\geq i(\alpha(l))+j(\alpha(l))-i(\alpha)$.
We see that
$\gamma:=\varepsilon_{i(\alpha)}+\varepsilon_{i(\alpha(l))+j(\alpha(l))-i(\alpha)}$
is a root and
$\gamma \geq \alpha$.
Thus $\gamma\in Y$ and $\hite(\gamma)=\hite(\alpha(l))$,
which contradicts the fact that $\alpha(l)\in M$.

By \eqref{eqn:h-hl},
all roots in $P_{\leq t_l}(\Lambda^{(\hite(\alpha))})$
with $\m$-weight $-a-2b$ ($-2a-2b$ if $t_{l+1}=j_l$ in Type $C$)
are in ${Y}(l)$.
Hence by Lemma \ref{ToricDeformation}
the linear combination with $\m$-weight $-a-b$ survives;
$$
P_{\leq t_{l+1}}(\Lambda^{(\hite(\alpha))})
\in
\lim_{t\to 0}\Ad(\lambda^\m(t))(\C 
P_{\leq t_l}(\Lambda^{(\hite(\alpha))})+\bigoplus_{\beta\in Y(l)}\C X_\beta).
$$

Similarly, by $t_{l+1}=t_l$ and the above table,
$$
P_{\leq t_{l+1}}(\Theta_k)=\lim_{t\to 0}\Ad(\lambda^\m(t))
(P_{\leq t_{l}}(\Theta_k))
\in \lim_{t\to 0}\Ad(\lambda^\m(t))(\fka_l).
$$

Clearly $\lim_{t\to 0}\Ad(\lambda^\m(t))(\C X_\alpha)
=\C X_\alpha$.
Hence in Case (a)
$$
\fka_{l+1}=\lim_{t\to 0}\Ad(\lambda^\m(t))(\fka_l).
$$

Next consider Case (b);
suppose that $t_{l+1}<\infty, t_l=\infty, \alpha(l-1)\notin M_2$.
In this case,  $s(\alpha(l))=\varepsilon_{t_{l+1}}+\varepsilon_{j_l}$,
and $t_{l+1}=j_l-1$ 
for Types $B, D$ and $t_{l+1}=j_l$
for Type $C$, respectively.

Again we put
$$
\m=-a\e_{t_{l+1}}-b\e_{j_l}\in \Z^n, 
$$
where $a>b>0$, and we show
\begin{equation}
\label{eq:Case3-b}
\fka_{l+1}=\lim_{t\to 0}\Ad(\lambda^\m(t))(\fka_l).
\end{equation}

For $\alpha\in M_2$ satisfying $s(\alpha)=s(\alpha(l))$,
$$
X_\alpha\in \lim_{t\to 0}\Ad(\lambda^\m(t))(\C 
P_{\leq t_l}(\Lambda^{(\hite(\alpha))}))
$$
as in Case (a).

For $\alpha\in M\setminus {Y}(l)$ and
$s(\alpha)\neq s(\alpha(l))$,
$$
P_{\leq t_{l+1}}(\Lambda^{(\hite(\alpha))})
\in
\lim_{t\to 0}\Ad(\lambda^\m(t))(\C 
P_{\leq t_l}(\Lambda^{(\hite(\alpha))})+\bigoplus_{\beta\in Y(l)}\C X_\beta)
$$
as in Case (a) as well.

In this case, ${\rm ht}(s(\alpha(l))=\min {\rm ht} Y$.
Hence the possible
$\m$-weights of roots appearing in $\Theta_k$
are $-a-b, -b, 0$ (Types $B, D$) and $-a-b, 0$ (Type $C$), respectively.
The roots with $\m$-weight $-a-b$ 
are exactly those roots appearing $P_{\leq t_{l+1}}(\Theta_k)$.
Hence
$$
P_{\leq t_{l+1}}(\Theta_k)=\lim_{t\to 0}\Ad(\lambda^\m(t))
(P_{\leq t_{l}}(\Theta_k))
\in \lim_{t\to 0}\Ad(\lambda^\m(t))(\fka_l).
$$

Since $\lim_{t\to 0}\Ad(\lambda^\m(t))(\C X_\alpha)
=\C X_\alpha$ is clear,
we see 
$$
\fka_{l+1}=\lim_{t\to 0}\Ad(\lambda^\m(t))(\fka_l)
$$
in Case (b) as well.

We have thus finished the proof of the theorem.
\end{proof}

\section{Case $G_2$}

In the exceptional types, we fix a Chevalley basis 
$$
\{ X_\alpha, H_i\,|\, \alpha\in \Delta,\, i=1,2,\ldots,n\}
$$
of $\fkg$ as in \cite[Proposition 4]{Kurtzke}.
In particular, we have
\begin{equation}
\label{eqn:ChevalleyBasis}
[X_{\alpha_i}, X_\beta]=(p+1)X_{\beta+\alpha_i}
\end{equation}
for non-simple $\beta\in \Delta^+$ with $\beta+\alpha_i\in \Delta^+$,
where $p$ is the nonnegative integer satisfying
$$
\beta-p\alpha_i\in \Delta^+,\quad \beta-(p+1)\alpha_i\notin \Delta^+.
$$
Note that $p=0$ in the case $E$.

In this section, let $\fkg$ be of $G_2$ type.
Then
the
$\alpha_1$-, $\alpha_2$-strings in $\Delta^+$ are as follows:
\begin{equation}
\label{CD:string}
\begin{CD}
@. @. @. 3\alpha_1+2\alpha_2\\
@. @. @. @AAA \\
\alpha_2 @>>> \alpha_1+\alpha_2 @>>> 2\alpha_1+\alpha_2 @>>> 3\alpha_1+\alpha_2\\
@. @AAA @. @.\\
@. \alpha_1 @. @. 
\end{CD}.
\end{equation}

The element $\Lambda=X_{\alpha_1}+X_{\alpha_2}$
is regular nilpotent, and
its centralizer equals
$$
J:=\C \Lambda \bigoplus \C X_{3\alpha_1+2\alpha_2}.
$$
In this case, by \eqref{CD:string}
there exists a unique $2$-dimensional abelian $\fkb$-ideal:
$$
K:=\C X_{3\alpha_1+\alpha_2} \bigoplus \C X_{3\alpha_1+2\alpha_2}.
$$
We have
$K=
\lim_{t\to 0}
\exp(t^{-1}\ad X_{2\alpha_1+\alpha_2})(J)$.

Thus Theorem \ref{MainTheorem} holds trivially, and
the $\fkg$-module $C_2$ is irreducible.

\section{Case $F_4$}

Let $\fkg$ be of $F_4$ type.
The Dynkin diagram is
$$
\xymatrix{
\alpha_1 \ar@{-}[r] & \alpha_2 \ar@{=>}[r] &
\alpha_3 \ar@{-}[r] & \alpha_4
}.
$$

\begin{lemma}
Let $\fka$ be a $4$-dimensional abelian 
$\fkb$-ideal.
Then
$$
\fka=\langle X_\alpha \,|
\alpha=2342, 1342, 1242, 1232
\rangle=K,
$$
where $2342$ means $2\alpha_1+3\alpha_2+4\alpha_3+2\alpha_4$,
etc.

Hence Theorem \ref{MainTheorem} holds trivially.
\end{lemma}

\begin{proof}
The maximal root is $2342$, whose height is $11$.
We know 
that $1342, 1242, 1232$ are the unique roots
with height $10, 9, 8$, respectively.
\end{proof}

The following is the list of
non-simple positive roots:

$$
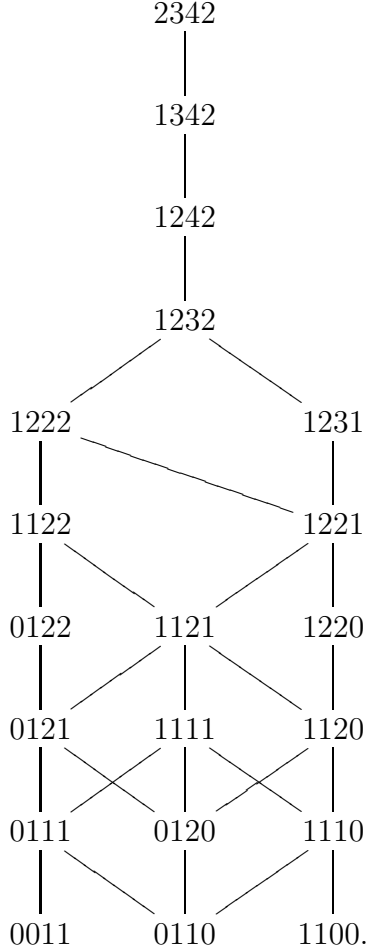
\xymatrix{
& 2342 \ar@{-}[d] & \\
& 1342 \ar@{-}[d] & \\
& 1242 \ar@{-}[d] & \\
& 1232 \ar@{-}[ld] \ar@{-}[rd] & \\
1222 \ar@{-}[d] \ar@{-}[rrd] & & 
1231 \ar@{-}[d]\\
1122 \ar@{-}[d] \ar@{-}[rd] &&
1221 \ar@{-}[d] \ar@{-}[ld]\\
0122 \ar@{-}[d] & 1121
\ar@{-}[ld] \ar@{-}[d] \ar@{-}[rd] &
1220 \ar@{-}[d]\\
0121 \ar@{-}[d] \ar@{-}[rd] &
1111 \ar@{-}[ld] \ar@{-}[d] \ar@{-}[rd] &
1120 \ar@{-}[ld] \ar@{-}[d]\\
0111 \ar@{-}[d] \ar@{-}[rd] &
0120 \ar@{-}[d] &
1110 \ar@{-}[ld] \ar@{-}[d]\\
0011 & 0110 & 1100.
}
$$

The element $\Lambda:=X_{\alpha_1}+X_{\alpha_2}+X_{\alpha_3}+X_{\alpha_4}$
is regular nilpotent,
and let $J:=\fkz_{\fkg}(\Lambda)$.

\begin{lemma}
$$
J=\langle
\Lambda,\, 2X_{0122}-X_{1121}+X_{1220},\, X_{1222}-X_{1231},\, X_{2342}
\rangle.
$$
\end{lemma}

\begin{proof}
Since
$[X_{\alpha_i}, X_{1222}]=\delta_{i3}X_{1232}$
and
$[X_{\alpha_i}, X_{1231}]=\delta_{i4}X_{1232}$,
we see that
$X_{1222}-X_{1231}\in J$.
Similarly, since
$[X_{\alpha_i}, X_{0122}]=\delta_{i1}X_{1122}$,
$[X_{\alpha_i}, X_{1121}]=2\delta_{i4}X_{1122}+\delta_{i2}X_{1221}$,
and
$[X_{\alpha_i}, X_{1220}]=\delta_{i4}X_{1221}$,
we see that
$2X_{0122}-X_{1121}+X_{1220}\in J$.
\end{proof}

\begin{proposition}
$$
(K=)\fka\in \overline{\Ad(G)J}.
$$
\end{proposition}

\begin{proof}
By considering heights, we see
\begin{eqnarray*}
&&\lim_{t\to 0}\exp(t^{-1}\ad(X_{1242}))(J)\\
&=&
\langle
2X_{0122}-X_{1121}+X_{1220},
X_{1222}-X_{1231},
X_{1342}, X_{2342}
\rangle=:\fka_1.
\end{eqnarray*}
We easily see
\begin{eqnarray*}
&&\lim_{t\to 0}\exp(t^{-1}\ad(X_{0121}))(\fka_1)\\
&=&
\langle
X_{1222}-X_{1231}, X_{1242},
X_{1342}, X_{2342}
\rangle=:\fka_2.
\end{eqnarray*}
Finally
\begin{eqnarray*}
&&\lim_{t\to 0}\exp(t^{-1}\ad(X_{\alpha_4}))(\fka_2)\\
&=&
\langle
X_{1232}, X_{1242},
X_{1342}, X_{2342}
\rangle=\fka.
\end{eqnarray*}
\end{proof}
\section{Case $E_6$}

Let $\fkg$ be of $E_6$ type.
The Dynkin diagram is
$$
\xymatrix{
\alpha_1 \ar@{-}[r] & \alpha_3 \ar@{-}[r] &
\alpha_4 \ar@{-}[r] \ar@{-}[d] & \alpha_5 \ar@{-}[r] &
\alpha_6\\
&& \alpha_2 &&
}.
$$

\begin{lemma}
There exist three $6$-dimensional abelian 
$\fkb$-ideals:

$
\fka_{i}=\fka'
\bigoplus \fka'_i\qquad (i=1,2,3),
$

\noindent
where
\begin{itemize}
\item
$\fka'=
\langle 
X_{\substack{12321\\ \phantom{12}2\phantom{21}}},
X_{\substack{12321\\ \phantom{12}1\phantom{21}}},
X_{\substack{12221\\ \phantom{12}1\phantom{21}}},
X_{\substack{11221\\ \phantom{11}1\phantom{21}}},
X_{\substack{12211\\ \phantom{12}1\phantom{11}}}
\rangle,
$
\item
$\fka'_1=\langle 
X_{\substack{01221\\ \phantom{01}1\phantom{21}}}
\rangle$,
\item
$\fka'_2=\langle 
X_{\substack{11211\\ \phantom{11}1\phantom{11}}}
\rangle$,
\item
$\fka'_3=\langle 
X_{\substack{12210\\ \phantom{12}1\phantom{10}}}
\rangle$.
\end{itemize}
\end{lemma}

\begin{proof}
We have the assertion by the following list
of non-simple positive roots:
$$
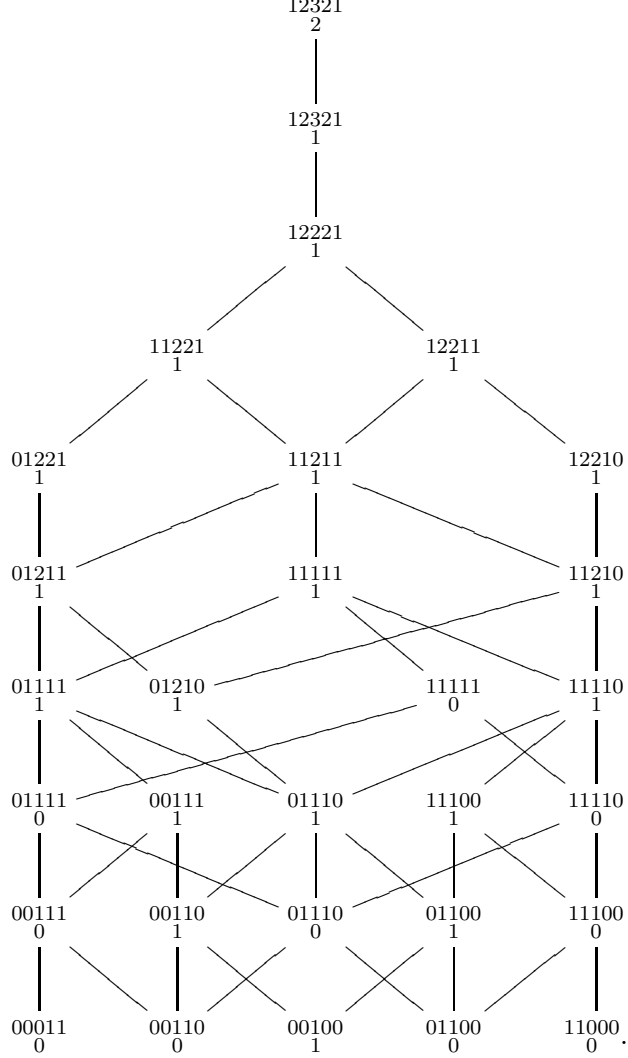
\xymatrix{
&& 
{\substack{12321\\ \phantom{12}2\phantom{21}}}
\ar@{-}[d] && \\
&&
{\substack{12321\\ \phantom{12}1\phantom{21}}}
\ar@{-}[d] && \\
&& {\substack{12221\\ \phantom{12}1\phantom{21}}}
\ar@{-}[rd] \ar@{-}[ld] &&\\
& {\substack{11221\\ \phantom{11}1\phantom{21}}}
\ar@{-}[rd] \ar@{-}[ld] & &
{\substack{12211\\ \phantom{12}1\phantom{11}}}
\ar@{-}[rd] \ar@{-}[ld] &\\
{\substack{01221\\ \phantom{01}1\phantom{21}}}
\ar@{-}[d] && {\substack{11211\\ \phantom{11}1\phantom{11}}}
\ar@{-}[rrd] \ar@{-}[lld] \ar@{-}[d] &&
{\substack{12210\\ \phantom{12}1\phantom{10}}}
\ar@{-}[d]\\
{\substack{01211\\ \phantom{01}1\phantom{11}}}
\ar@{-}[d] \ar@{-}[rd]
&&
{\substack{11111\\ \phantom{11}1\phantom{11}}}
\ar@{-}[lld]
\ar@{-}[rd]
\ar@{-}[rrd]
&&
{\substack{11210\\ \phantom{11}1\phantom{10}}}
\ar@{-}[d]
\ar@{-}[llld]\\
{\substack{01111\\ \phantom{01}1\phantom{11}}}
\ar@{-}[d] \ar@{-}[rrd]
\ar@{-}[rd]
&
{\substack{01210\\ \phantom{01}1\phantom{10}}}
\ar@{-}[rd]
&&
{\substack{11111\\ \phantom{11}0\phantom{11}}}
\ar@{-}[rd] \ar@{-}[llld]
&
{\substack{11110\\ \phantom{11}1\phantom{10}}}
\ar@{-}[d] \ar@{-}[lld]
\ar@{-}[ld]\\
{\substack{01111\\ \phantom{01}0\phantom{11}}}
\ar@{-}[d] \ar@{-}[rrd]
&
{\substack{00111\\ \phantom{00}1\phantom{11}}}
\ar@{-}[d] 
\ar@{-}[ld]
&
{\substack{01110\\ \phantom{01}1\phantom{10}}}
\ar@{-}[d] 
\ar@{-}[ld]
\ar@{-}[rd] 
&
{\substack{11100\\ \phantom{11}1\phantom{00}}}
\ar@{-}[d] 
\ar@{-}[rd] 
&
{\substack{11110\\ \phantom{11}0\phantom{10}}}
\ar@{-}[d] 
\ar@{-}[lld] \\
{\substack{00111\\ \phantom{00}0\phantom{11}}}
\ar@{-}[d] 
\ar@{-}[rd] 
&
{\substack{00110\\ \phantom{00}1\phantom{10}}}
\ar@{-}[d] 
\ar@{-}[rd] 
&
{\substack{01110\\ \phantom{01}0\phantom{10}}}
\ar@{-}[ld] 
\ar@{-}[rd] 
&
{\substack{01100\\ \phantom{01}1\phantom{00}}}
\ar@{-}[ld] 
\ar@{-}[d] 
&
{\substack{11100\\ \phantom{11}0\phantom{00}}}
\ar@{-}[ld] 
\ar@{-}[d] \\
{\substack{00011\\ \phantom{00}0\phantom{11}}}
&
{\substack{00110\\ \phantom{00}0\phantom{10}}}
&
{\substack{00100\\ \phantom{00}1\phantom{00}}}
&
{\substack{01100\\ \phantom{01}0\phantom{00}}}
&
{\substack{11000\\ \phantom{11}0\phantom{00}}}.
}
$$
\end{proof}

The element $\Lambda:=\sum_{i=1}^6 X_{\alpha_i}$ is regular nilpotent,
and let $J:=\fkz_{\fkg}(\Lambda)$.
By the list above, we have the following lemma:

\begin{lemma}
The following form a basis of $J$:
\begin{itemize}
\item
$f_1:=\Lambda$,
\item
$f_4:=X_{\substack{01111\\ \phantom{01}0\phantom{11}}}
-X_{\substack{00111\\ \phantom{00}1\phantom{11}}}
-X_{\substack{11110\\ \phantom{11}0\phantom{10}}}
+X_{\substack{11100\\ \phantom{11}1\phantom{00}}}$,
\item
$f_5:=X_{\substack{01111\\ \phantom{01}1\phantom{11}}}
-X_{\substack{01210\\ \phantom{01}1\phantom{10}}}
+X_{\substack{11110\\ \phantom{11}1\phantom{10}}}
-2X_{\substack{11111\\ \phantom{11}0\phantom{11}}}$,
\item
$f_7:=X_{\substack{01221\\ \phantom{01}1\phantom{21}}}
-X_{\substack{11211\\ \phantom{11}1\phantom{11}}}
+X_{\substack{12210\\ \phantom{12}1\phantom{10}}}$,
\item
$f_8:=X_{\substack{11221\\ \phantom{11}1\phantom{21}}}
-X_{\substack{12211\\ \phantom{12}1\phantom{11}}}$,
\item
$f_{11}:=X_{\substack{12321\\ \phantom{12}2\phantom{21}}}$.
\end{itemize}
\end{lemma}

\begin{proposition}
Proposition \ref{prop:JtoK} and
Theorem \ref{MainTheorem} hold.
In particular,
$$
\fka_{i}\in
\overline{\Ad(G)J}\qquad (i=1,2,3).
$$
\end{proposition}

\begin{proof}
Since
\begin{eqnarray*}
\ad(X_{\substack{11111\\ \phantom{11}0\phantom{11}}})
(\Lambda)
&=&
-\sum_{i=1}^6 [X_{\alpha_i}, 
X_{\substack{11111\\ \phantom{11}0\phantom{11}}}]
=
-X_{\substack{11111\\ \phantom{11}1\phantom{11}}},\\
\ad(X_{\substack{01210\\ \phantom{01}1\phantom{10}}})
(\Lambda)
&=&
-\sum_{i=1}^6 [X_{\alpha_i}, 
X_{\substack{01210\\ \phantom{01}1\phantom{10}}}]
=
-X_{\substack{11210\\ \phantom{11}1\phantom{10}}}
-X_{\substack{01211\\ \phantom{01}1\phantom{11}}},
\end{eqnarray*}
there exist $a, b$ such that
\begin{eqnarray*}
\ad(a X_{\substack{11111\\ \phantom{11}0\phantom{11}}}
+
b X_{\substack{01210\\ \phantom{01}1\phantom{10}}})
(\Lambda)
&=&
c_1X_{\substack{01211\\ \phantom{01}1\phantom{11}}}
+
c_2X_{\substack{11111\\ \phantom{11}1\phantom{11}}}
+
c_3X_{\substack{11210\\ \phantom{11}1\phantom{10}}},\\
\ad(a X_{\substack{11111\\ \phantom{11}0\phantom{11}}}
+
b X_{\substack{01210\\ \phantom{01}1\phantom{10}}})
(f_4)
&=&
c_4 X_{\substack{12221\\ \phantom{12}1\phantom{21}}},\\
\ad(a X_{\substack{11111\\ \phantom{11}0\phantom{11}}}
+
b X_{\substack{01210\\ \phantom{01}1\phantom{10}}})
(f_5)
&=&
c_5 X_{\substack{12321\\ \phantom{12}1\phantom{21}}}
\end{eqnarray*}
with $c_1,c_2,\ldots,c_5\neq 0$.
Hence
$$
\lim_{t\to 0}\exp t^{-1}
\ad(a X_{\substack{11111\\ \phantom{11}0\phantom{11}}}
+
b X_{\substack{01210\\ \phantom{01}1\phantom{10}}})
(J)=:K
$$
satisfies the condition in Proposition \ref{prop:JtoK}.

Then it is easy to see
$$
\lim_{t\to 0}\exp t^{-1}
\ad(
X_{\substack{00110\\ \phantom{00}0\phantom{10}}}
)(K)
=\fka'\bigoplus \langle f_7\rangle.
$$


We have
\begin{eqnarray*}
\lim_{t\to 0}
\Ad(\lambda_1(t))(
\fka'\bigoplus \langle f_7\rangle)
&=&
\fka_1,\\
\lim_{t\to 0}
\Ad(\lambda_1^{-1}(t)\lambda_6^{-1}(t))(
\fka'\bigoplus \langle f_7\rangle)
&=&
\fka_2,\\
\lim_{t\to 0}
\Ad(\lambda_6(t))(
\fka'\bigoplus \langle f_7\rangle)
&=&
\fka_3.
\end{eqnarray*}
Hence we have proved 
Theorem \ref{MainTheorem}, i.e.,
$$
\fka_1,\fka_2,\fka_3
\in 
\overline{\Ad(G)K}.
$$
\end{proof}

\section{Case $E_7$}

Let $\fkg$ be of $E_7$ type.
The Dynkin diagram is
$$
\xymatrix{
\alpha_1 \ar@{-}[r] & \alpha_3 \ar@{-}[r] &
\alpha_4 \ar@{-}[r] \ar@{-}[d] & \alpha_5 \ar@{-}[r] &
\alpha_6 \ar@{-}[r] &
\alpha_7\\
&& \alpha_2 &&
}.
$$

\begin{lemma}
\label{E7:fka}
There exist three $7$-dimensional abelian 
$\fkb$-ideals:

$
\fka_{i}=
\langle 
X_{\substack{234321\\ \phantom{23}2\phantom{321}}},
X_{\substack{134321\\ \phantom{13}2\phantom{321}}},
X_{\substack{124321\\ \phantom{12}2\phantom{321}}},
X_{\substack{123321\\ \phantom{12}2\phantom{321}}},
X_{\substack{123221\\ \phantom{12}2\phantom{221}}}
\rangle
\bigoplus \fka'_i\qquad (i=1,2,3),
$

\noindent
where
\begin{itemize}
\item
$\fka'_1=\langle 
X_{\substack{123211\\ \phantom{12}2\phantom{211}}},
X_{\substack{123210\\ \phantom{12}2\phantom{210}}}
\rangle$,
\item
$\fka'_2=\langle 
X_{\substack{123321\\ \phantom{12}1\phantom{321}}},
X_{\substack{123211\\ \phantom{12}2\phantom{211}}}
\rangle$,
\item
$\fka'_3=\langle 
X_{\substack{123321\\ \phantom{12}1\phantom{321}}},
X_{\substack{123221\\ \phantom{12}1\phantom{221}}}
\rangle$.
\end{itemize}
\end{lemma}

\begin{proof}
The following is the list of positive roots with height greater than $10$:
$$
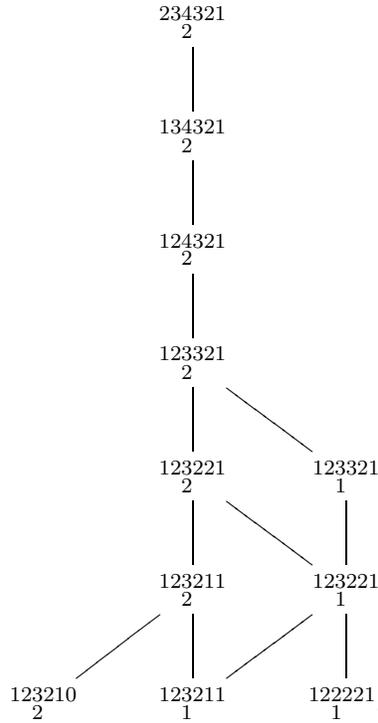
\xymatrix{
&
{\substack{234321\\ \phantom{23}2\phantom{321}}}
\ar@{-}[d] & \\
&
{\substack{134321\\ \phantom{13}2\phantom{321}}}
\ar@{-}[d] & \\
&
{\substack{124321\\ \phantom{12}2\phantom{321}}}
\ar@{-}[d] & \\
&
{\substack{123321\\ \phantom{12}2\phantom{321}}}
\ar@{-}[d] 
\ar@{-}[rd]
& \\
&
{\substack{123221\\ \phantom{12}2\phantom{221}}}
\ar@{-}[d] 
\ar@{-}[rd]
& 
{\substack{123321\\ \phantom{12}1\phantom{321}}}
\ar@{-}[d] 
\\
&
{\substack{123211\\ \phantom{12}2\phantom{211}}}
\ar@{-}[d] 
\ar@{-}[ld]
& 
{\substack{123221\\ \phantom{12}1\phantom{221}}}
\ar@{-}[d] 
\ar@{-}[ld]
\\
{\substack{123210\\ \phantom{12}2\phantom{210}}}
&
{\substack{123211\\ \phantom{12}1\phantom{211}}}
&
{\substack{122221\\ \phantom{12}1\phantom{221}}}.
}
$$
Hence the assertion holds.
\end{proof}

The element $\Lambda:=\sum_{i=1}^7 X_{\alpha_i}$
is regular nilpotent,
and let $J:=\fkz_{\fkg}(\Lambda)$.

\begin{lemma}
The following form a basis of $J$:
\begin{itemize}
\item
$f_1:=\Lambda$,
\item
$f_5:=X_{\substack{012100\\ \phantom{01}1\phantom{100}}}
-X_{\substack{111100\\ \phantom{11}1\phantom{100}}}
-X_{\substack{011110\\ \phantom{01}1\phantom{110}}}
+2X_{\substack{111110\\ \phantom{11}0\phantom{110}}}
-2X_{\substack{011111\\ \phantom{01}0\phantom{111}}}
+3X_{\substack{001111\\ \phantom{00}1\phantom{111}}}
$,
\item
$f_7:=X_{\substack{122100\\ \phantom{12}1\phantom{100}}}
-X_{\substack{112110\\ \phantom{11}1\phantom{110}}}
+X_{\substack{012210\\ \phantom{01}1\phantom{210}}}
-X_{\substack{012111\\ \phantom{01}1\phantom{111}}}
+2X_{\substack{111111\\ \phantom{11}1\phantom{111}}}
$,
\item
$f_9:=X_{\substack{122111\\ \phantom{12}1\phantom{111}}}
-X_{\substack{112211\\ \phantom{11}1\phantom{211}}}
+X_{\substack{012221\\ \phantom{01}1\phantom{221}}}$,
\item
$f_{11}:=X_{\substack{123210\\ \phantom{12}2\phantom{210}}}
-X_{\substack{123211\\ \phantom{12}1\phantom{211}}}
+X_{\substack{122221\\ \phantom{12}1\phantom{221}}}
$,
\item
$f_{13}:=X_{\substack{123221\\ \phantom{12}2\phantom{221}}}
-X_{\substack{123321\\ \phantom{12}1\phantom{321}}}
$,
\item
$f_{17}:=X_{\substack{234321\\ \phantom{23}2\phantom{321}}}$.
\end{itemize}
\end{lemma}

\begin{proof}
From the proof of Lemma \ref{E7:fka},
we see $f_{17}, f_{13}, f_{11}\in J$.

The following is the list of positive roots with height $10,9$:
$$
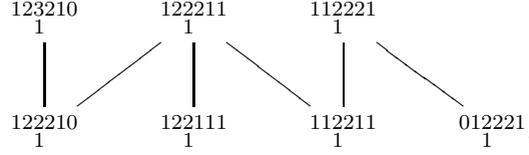
\xymatrix{
{\substack{123210\\ \phantom{12}1\phantom{210}}}
\ar@{-}[d]
&
{\substack{122211\\ \phantom{12}1\phantom{211}}}
\ar@{-}[d]
\ar@{-}[ld]
\ar@{-}[rd]
&
{\substack{112221\\ \phantom{11}1\phantom{221}}}
\ar@{-}[d]
\ar@{-}[rd]
&\\
{\substack{122210\\ \phantom{12}1\phantom{210}}}
&
{\substack{122111\\ \phantom{12}1\phantom{111}}}
&
{\substack{112211\\ \phantom{11}1\phantom{211}}}
&
{\substack{012221\\ \phantom{01}1\phantom{221}}}.
}
$$
Hence $f_9\in J$.

The following is the list of positive roots with height $8,7$:
$$
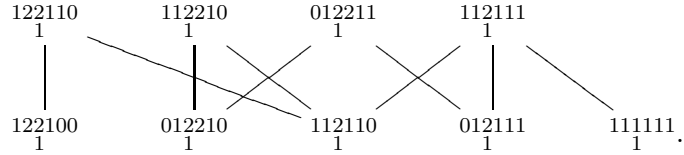
\xymatrix{
{\substack{122110\\ \phantom{12}1\phantom{110}}}
\ar@{-}[d]
\ar@{-}[rrd]
&
{\substack{112210\\ \phantom{11}1\phantom{210}}}
\ar@{-}[d]
\ar@{-}[rd]
&
{\substack{012211\\ \phantom{01}1\phantom{211}}}
\ar@{-}[ld]
\ar@{-}[rd]
&
{\substack{112111\\ \phantom{11}1\phantom{111}}}
\ar@{-}[ld]
\ar@{-}[rd]
\ar@{-}[d]
&\\
{\substack{122100\\ \phantom{12}1\phantom{100}}}
&
{\substack{012210\\ \phantom{01}1\phantom{210}}}
&
{\substack{112110\\ \phantom{11}1\phantom{110}}}
&
{\substack{012111\\ \phantom{01}1\phantom{111}}}
&
{\substack{111111\\ \phantom{11}1\phantom{111}}}.
}
$$
Hence $f_7\in J$.

The following is the list of positive roots with height $6,5$:
$$
\xymatrix{
{\substack{112100\\ \phantom{11}1\phantom{100}}}
\ar@{-}[d]
\ar@{-}[rd]
&
{\substack{012110\\ \phantom{01}1\phantom{110}}}
\ar@{-}[ld]
\ar@{-}[rrd]
&
{\substack{111110\\ \phantom{11}1\phantom{110}}}
\ar@{-}[ld]
\ar@{-}[d]
\ar@{-}[rd]
&
{\substack{111111\\ \phantom{11}0\phantom{111}}}
\ar@{-}[ld]
\ar@{-}[rd]
&
{\substack{011111\\ \phantom{01}1\phantom{111}}}
\ar@{-}[ld]
\ar@{-}[rd]
\ar@{-}[d]
&\\
{\substack{012100\\ \phantom{01}1\phantom{100}}}
&
{\substack{111100\\ \phantom{11}1\phantom{100}}}
&
{\substack{111110\\ \phantom{11}0\phantom{110}}}
&
{\substack{011110\\ \phantom{01}1\phantom{110}}}
&
{\substack{011111\\ \phantom{01}0\phantom{111}}}
&
{\substack{001111\\ \phantom{00}1\phantom{111}}}.
}
$$
Hence $f_5\in J$.
\end{proof}

\begin{proposition}
Proposition \ref{prop:JtoK} and
Theorem \ref{MainTheorem} hold.
In particular,
$$
\fka_{i}\in
\overline{\Ad(G)J}\qquad (i=1,2,3).
$$
\end{proposition}

\begin{proof}
By considering heights, we see
\begin{eqnarray*}
&&\lim_{t\to 0}\exp(t^{-1}
\ad(X_{\substack{124321\\ \phantom{12}2\phantom{321}}}
))(J)\\
&=&
\langle
f_5,f_7,f_9,f_{11},f_{13},
X_{\substack{134321\\ \phantom{13}2\phantom{321}}},
f_{17}
\rangle=:J_1.
\end{eqnarray*}
Since 
$f_{17}\in J_1$, we have
\begin{eqnarray*}
&&\lim_{t\to 0}\exp(t^{-1}
\ad(X_{\substack{123210\\ \phantom{12}1\phantom{210}}}
))(J_1)\\
&=&
\langle
f_7,f_9,f_{11},f_{13},
X_{\substack{124321\\ \phantom{12}2\phantom{321}}},
X_{\substack{134321\\ \phantom{13}2\phantom{321}}},
f_{17}
\rangle=:J_2.
\end{eqnarray*}
Since 
$X_{\substack{134321\\ \phantom{13}2\phantom{321}}}\in J_2$, we have
\begin{eqnarray*}
&&\lim_{t\to 0}\exp(t^{-1}
\ad(X_{\substack{012210\\ \phantom{01}1\phantom{210}}}
))(J_2)\\
&=&
\langle
f_9,f_{11},f_{13},
X_{\substack{123321\\ \phantom{12}2\phantom{321}}},
X_{\substack{124321\\ \phantom{12}2\phantom{321}}},
X_{\substack{134321\\ \phantom{13}2\phantom{321}}},
f_{17}
\rangle=:J_3.
\end{eqnarray*}
As in Case $E_6$,
there exist $a, b$ such that
\begin{eqnarray*}
&&
\lim_{t\to 0}\exp(t^{-1}
\ad(a X_{\substack{001100\\ \phantom{00}1\phantom{100}}}
+ b
X_{\substack{111000\\ \phantom{11}0\phantom{000}}}
))(J_3)\\
&=&
\langle
X_\alpha\,|\, \hite(\alpha)\geq 14
\rangle
\bigoplus
\langle
f_{13}, \Lambda^{(12)}, f_{11}
\rangle
=K
\end{eqnarray*}
satisfies the condition in Proposition \ref{prop:JtoK}.

Then
$$
\lim_{t\to 0}\Ad(\lambda_2^{-1}(t))(K)
=\fka_1.
$$

We have
$$
\lim_{t\to 0}\exp(t^{-1}
\ad(X_{\substack{000011\\ \phantom{00}0\phantom{011}}}))
(K)
=
\langle
X_\alpha\,|\, \hite(\alpha)\geq 13
\rangle
\bigoplus
\langle \Lambda^{(12)}\rangle,
$$
and then
\begin{eqnarray*}
\lim_{t\to 0}\Ad(\lambda_2^{-1}(t))
(\langle
X_\alpha\,|\, \hite(\alpha)\geq 13
\rangle
\bigoplus
\langle \Lambda^{(12)}\rangle)
&=&
\fka_2,\\
\lim_{t\to 0}\Ad(\lambda_2(t))
(\langle
X_\alpha\,|\, \hite(\alpha)\geq 13
\rangle
\bigoplus
\langle \Lambda^{(12)}\rangle)
&=&
\fka_3.
\end{eqnarray*}
We have thus proved
$$
\fka_1,\fka_2,\fka_3
\in 
\overline{\Ad(G)K}.
$$
\end{proof}

\section{Case $E_8$}

Let $\fkg$ be of $E_8$ type.
The Dynkin diagram is
$$
\xymatrix{
\alpha_1 \ar@{-}[r] & \alpha_3 \ar@{-}[r] &
\alpha_4 \ar@{-}[r] \ar@{-}[d] & \alpha_5 \ar@{-}[r] &
\alpha_6 \ar@{-}[r] &
\alpha_7 \ar@{-}[r] &
\alpha_8\\
&& \alpha_2 &&
}.
$$

\begin{lemma}
\label{E8:fka}
There exist two $8$-dimensional abelian 
$\fkb$-ideals:

$\fka_1=\fka'\bigoplus \langle X_{\substack{1354321\\ \phantom{13}3\phantom{4321}}}\rangle,
\quad
\fka_2=\fka'\bigoplus \langle X_{\substack{2454321\\ \phantom{24}2\phantom{4321}}}\rangle,
$ where
\newline
$\fka':=
\langle 
X_{\substack{2465432\\ \phantom{24}3\phantom{5432}}},
X_{\substack{2465431\\ \phantom{24}3\phantom{5431}}},
X_{\substack{2465421\\ \phantom{24}3\phantom{5421}}},
X_{\substack{2465321\\ \phantom{24}3\phantom{5321}}},
X_{\substack{2464321\\ \phantom{24}3\phantom{4321}}},
X_{\substack{2454321\\ \phantom{24}3\phantom{4321}}},
X_{\substack{2354321\\ \phantom{23}3\phantom{4321}}}
\rangle.
$
\end{lemma}

\begin{proof}
The following is the list of positive roots 
with heights greater than $21$:
$$
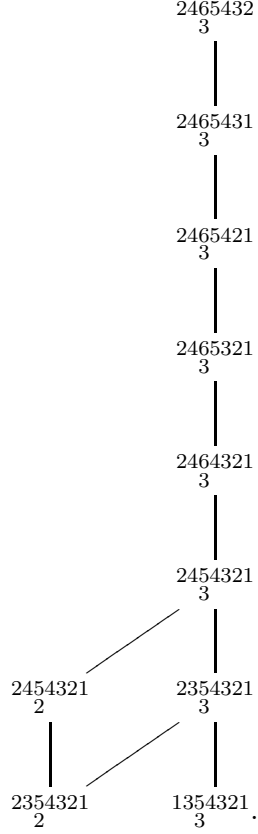
\xymatrix{
& {\substack{2465432\\ \phantom{24}3\phantom{5432}}}
\ar@{-}[d]
\\
& {\substack{2465431\\ \phantom{24}3\phantom{5431}}}
\ar@{-}[d]
\\
& {\substack{2465421\\ \phantom{24}3\phantom{5421}}}
\ar@{-}[d]
\\
& {\substack{2465321\\ \phantom{24}3\phantom{5321}}}
\ar@{-}[d]
\\
& {\substack{2464321\\ \phantom{24}3\phantom{4321}}}
\ar@{-}[d]
\\
& {\substack{2454321\\ \phantom{24}3\phantom{4321}}}
\ar@{-}[d]
\ar@{-}[ld]
\\
{\substack{2454321\\ \phantom{24}2\phantom{4321}}}
\ar@{-}[d]
& 
{\substack{2354321\\ \phantom{23}3\phantom{4321}}}
\ar@{-}[d]
\ar@{-}[ld]
\\
{\substack{2354321\\ \phantom{23}2\phantom{4321}}}
&
{\substack{1354321\\ \phantom{13}3\phantom{4321}}}.
}
$$
Hence the assertion holds.
\end{proof}

The element $\Lambda:=\sum_{i=1}^8 X_{\alpha_i}$
is regular nilpotent,
and let $J:=\fkz_{\fkg}(\Lambda)$.

\begin{lemma}
The following form a basis of $J$:
\begin{itemize}
\item
$f_1:=\Lambda$,
\item
$f_7:=X_{\substack{1221000\\ \phantom{12}1\phantom{1000}}}
-X_{\substack{1121100\\ \phantom{11}1\phantom{1100}}}
+X_{\substack{0122100\\ \phantom{01}1\phantom{2100}}}
-X_{\substack{0121110\\ \phantom{01}1\phantom{1110}}}
+2X_{\substack{1111110\\ \phantom{11}1\phantom{1110}}}
-2X_{\substack{1111111\\ \phantom{11}0\phantom{1111}}}
+X_{\substack{0111111\\ \phantom{01}1\phantom{1111}}}$,
\item
$f_{11}:=X_{\substack{1232100\\ \phantom{12}2\phantom{2100}}}
-X_{\substack{1232110\\ \phantom{12}1\phantom{2110}}}
+X_{\substack{1222210\\ \phantom{12}1\phantom{2210}}}
+X_{\substack{1222111\\ \phantom{12}1\phantom{2111}}}
-2X_{\substack{1122211\\ \phantom{11}1\phantom{2211}}}
+2X_{\substack{0122221\\ \phantom{01}1\phantom{2221}}}$,
\item
$f_{13}:=
X_{\substack{1222221\\ \phantom{12}1\phantom{2221}}}
-X_{\substack{1232211\\ \phantom{12}1\phantom{2211}}}
+X_{\substack{1233210\\ \phantom{12}1\phantom{3210}}}
-X_{\substack{1232210\\ \phantom{12}2\phantom{2210}}}
+2X_{\substack{1232111\\ \phantom{12}2\phantom{2111}}}$,
\item
$f_{17}:=X_{\substack{2343210\\ \phantom{23}2\phantom{3210}}}
-X_{\substack{1343211\\ \phantom{13}2\phantom{3211}}}
+X_{\substack{1243221\\ \phantom{12}2\phantom{3221}}}
-X_{\substack{1233321\\ \phantom{12}2\phantom{3321}}}$,
\item
$f_{19}:=X_{\substack{2343221\\ \phantom{23}2\phantom{3221}}}
-X_{\substack{1343321\\ \phantom{13}2\phantom{3321}}}
+X_{\substack{1244321\\ \phantom{12}2\phantom{4321}}}$,
\item
$f_{23}:=X_{\substack{2454321\\ \phantom{24}2\phantom{4321}}}
-X_{\substack{2354321\\ \phantom{23}3\phantom{4321}}}$,
\item
$f_{29}:=X_{\substack{2465432\\ \phantom{24}3\phantom{5432}}}$.
\end{itemize}
\end{lemma}

\begin{proof}
By the proof of Lemma \ref{E8:fka},
$f_{23}\in J$.

The following is the list of roots with height $20, 19$:
$$
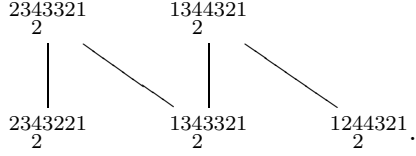
\xymatrix{
{\substack{2343321\\ \phantom{23}2\phantom{3321}}}
\ar@{-}[d]
\ar@{-}[rd]
&
{\substack{1344321\\ \phantom{13}2\phantom{4321}}}
\ar@{-}[d]
\ar@{-}[rd]
&
\\
{\substack{2343221\\ \phantom{23}2\phantom{3221}}}
&
{\substack{1343321\\ \phantom{13}2\phantom{3321}}}
&
{\substack{1244321\\ \phantom{12}2\phantom{4321}}}.
}
$$
Hence $f_{19}\in J$.
The following is the list of roots with height $18, 17$:
$$
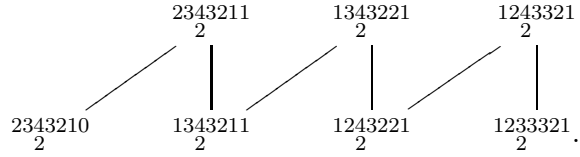
\xymatrix{
&
{\substack{2343211\\ \phantom{23}2\phantom{3211}}}
\ar@{-}[d]
\ar@{-}[ld]
&
{\substack{1343221\\ \phantom{13}2\phantom{3221}}}
\ar@{-}[d]
\ar@{-}[ld]
&
{\substack{1243321\\ \phantom{12}2\phantom{3321}}}
\ar@{-}[d]
\ar@{-}[ld]
\\
{\substack{2343210\\ \phantom{23}2\phantom{3210}}}
&
{\substack{1343211\\ \phantom{13}2\phantom{3211}}}
&
{\substack{1243221\\ \phantom{12}2\phantom{3221}}}
&
{\substack{1233321\\ \phantom{12}2\phantom{3321}}}.
}
$$
Hence $f_{17}\in J$.

The following is the list of roots with height $14, 13$:
$$
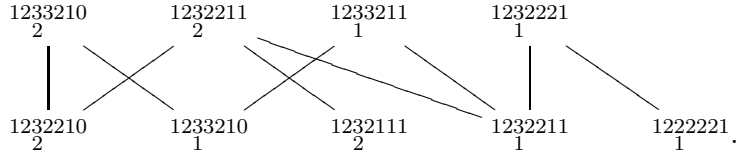
\xymatrix{
{\substack{1233210\\ \phantom{12}2\phantom{3210}}}
\ar@{-}[d]
\ar@{-}[rd]
&
{\substack{1232211\\ \phantom{12}2\phantom{2211}}}
\ar@{-}[ld]
\ar@{-}[rd]
\ar@{-}[rrd]
&
{\substack{1233211\\ \phantom{12}1\phantom{3211}}}
\ar@{-}[ld]
\ar@{-}[rd]
&
{\substack{1232221\\ \phantom{12}1\phantom{2221}}}
\ar@{-}[d]
\ar@{-}[rd]
&
\\
{\substack{1232210\\ \phantom{12}2\phantom{2210}}}
&
{\substack{1233210\\ \phantom{12}1\phantom{3210}}}
&
{\substack{1232111\\ \phantom{12}2\phantom{2111}}}
&
{\substack{1232211\\ \phantom{12}1\phantom{2211}}}
&
{\substack{1222221\\ \phantom{12}1\phantom{2221}}}.
}
$$
Hence $f_{13}\in J$.

The following is the list of roots with height $12, 11$:
$$
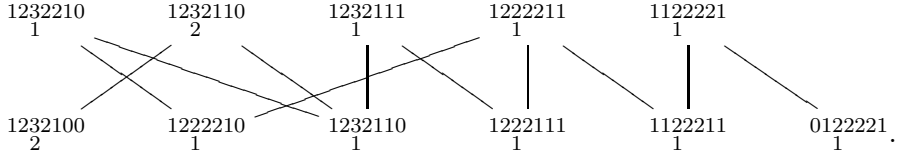
\xymatrix{
{\substack{1232210\\ \phantom{12}1\phantom{2210}}}
\ar@{-}[rd]
\ar@{-}[rrd]
&
{\substack{1232110\\ \phantom{12}2\phantom{2110}}}
\ar@{-}[rd]
\ar@{-}[ld]
&
{\substack{1232111\\ \phantom{12}1\phantom{2111}}}
\ar@{-}[rd]
\ar@{-}[d]
&
{\substack{1222211\\ \phantom{12}1\phantom{2211}}}
\ar@{-}[rd]
\ar@{-}[d]
\ar@{-}[lld]
&
{\substack{1122221\\ \phantom{11}1\phantom{2221}}}
\ar@{-}[rd]
\ar@{-}[d]
&
\\
{\substack{1232100\\ \phantom{12}2\phantom{2100}}}
&
{\substack{1222210\\ \phantom{12}1\phantom{2210}}}
&
{\substack{1232110\\ \phantom{12}1\phantom{2110}}}
&
{\substack{1222111\\ \phantom{12}1\phantom{2111}}}
&
{\substack{1122211\\ \phantom{11}1\phantom{2211}}}
&
{\substack{0122221\\ \phantom{01}1\phantom{2221}}}.
}
$$
Hence $f_{11}\in J$.

The following is the list of roots with height $8,7$:
{\tiny
$$
\xymatrix{
{\substack{1221100\\ \phantom{12}1\phantom{1100}}}
\ar@{-}[rd]
\ar@{-}[d]\!
&\!
{\substack{1122100\\ \phantom{11}1\phantom{2100}}}
\ar@{-}[rd]
\ar@{-}[d]\!
&\!
{\substack{1121110\\ \phantom{11}1\phantom{1110}}}
\ar@{-}[rd]
\ar@{-}[ld]
\ar@{-}[rrd]\!
&\!
{\substack{0122110\\ \phantom{01}1\phantom{2110}}}
\ar@{-}[d]
\ar@{-}[ld]\!
&\!
{\substack{0121111\\ \phantom{01}1\phantom{1111}}}
\ar@{-}[rrd]
\ar@{-}[ld]\!
&\!
{\substack{1111111\\ \phantom{11}1\phantom{1111}}}
\ar@{-}[d]
\ar@{-}[ld]\!
&\!
\\
{\substack{1221000\\ \phantom{12}1\phantom{1000}}}\!
&\!
{\substack{1121100\\ \phantom{11}1\phantom{1100}}}\!
&\!
{\substack{0122100\\ \phantom{01}1\phantom{2100}}}\!
&\!
{\substack{0121110\\ \phantom{01}1\phantom{1110}}}\!
&\!
{\substack{1111110\\ \phantom{11}1\phantom{1110}}}\!
&\!
{\substack{1111111\\ \phantom{11}0\phantom{1111}}}\!
&\!\!
{\substack{0111111\\ \phantom{01}1\phantom{1111}}}.
}
$$
}
Hence $f_{7}\in J$.
\end{proof}

\begin{proposition}
Proposition \ref{prop:JtoK} and
Theorem \ref{MainTheorem} hold.
In particular,
$$
\fka_{i}\in
\overline{\Ad(G)J}\qquad (i=1,2).
$$
\end{proposition}

\begin{proof}
By considering height, we see
\begin{eqnarray*}
&&\lim_{t\to 0}\exp(t^{-1}
\ad(X_{\substack{2465421\\ \phantom{24}3\phantom{5421}}}
))(J)\\
&=&
\langle
f_7,f_{11},f_{13},f_{17}, f_{19}, f_{23},
X_{\substack{2465431\\ \phantom{24}3\phantom{5431}}},
f_{29}
\rangle=:J_1,
\end{eqnarray*}
and
\begin{eqnarray*}
&&\lim_{t\to 0}\exp(t^{-1}
\ad(X_{\substack{2343321\\ \phantom{23}2\phantom{3321}}}
))(J_1)\\
&=&
\langle
f_{11},f_{13},f_{17}, f_{19}, f_{23},
X_{\substack{2465421\\ \phantom{24}3\phantom{5421}}},
X_{\substack{2465431\\ \phantom{24}3\phantom{5431}}},
f_{29}
\rangle=:J_2.
\end{eqnarray*}
Since $X_{\substack{2465431\\ \phantom{24}3\phantom{5431}}}
\in J_2$,
\begin{eqnarray*}
&&\lim_{t\to 0}\exp(t^{-1}
\ad(X_{\substack{1232100\\ \phantom{12}2\phantom{2100}}}
))(J_2)\\
&=&
\langle
f_{13},f_{17}, f_{19}, f_{23},
X_{\substack{2465321\\ \phantom{24}3\phantom{5321}}},
X_{\substack{2465421\\ \phantom{24}3\phantom{5421}}},
X_{\substack{2465431\\ \phantom{24}3\phantom{5431}}},
f_{29}
\rangle=:J_3.
\end{eqnarray*}
Since $f_{29}
\in J_3$,
\begin{eqnarray*}
&&\lim_{t\to 0}\exp(t^{-1}
\ad(X_{\substack{1232111\\ \phantom{12}2\phantom{2111}}}
))(J_3)\\
&=&
\langle
f_{17}, f_{19}, f_{23},
X_{\substack{2464321\\ \phantom{24}3\phantom{4321}}},
X_{\substack{2465321\\ \phantom{24}3\phantom{5321}}},
X_{\substack{2465421\\ \phantom{24}3\phantom{5421}}},
X_{\substack{2465431\\ \phantom{24}3\phantom{5431}}},
f_{29}
\rangle=:J_4.
\end{eqnarray*}
Since $X_{\substack{2465321\\ \phantom{24}3\phantom{5321}}}
\in J_4$,
\begin{eqnarray*}
&&\lim_{t\to 0}\exp(t^{-1}
\ad(X_{\substack{1111110\\ \phantom{11}1\phantom{1110}}}
))(J_4)\\
&=&
\langle
f_{19}, f_{23},
X_{\substack{2454321\\ \phantom{24}3\phantom{4321}}},
X_{\substack{2464321\\ \phantom{24}3\phantom{4321}}},
X_{\substack{2465321\\ \phantom{24}3\phantom{5321}}},
X_{\substack{2465421\\ \phantom{24}3\phantom{5421}}},
X_{\substack{2465431\\ \phantom{24}3\phantom{5431}}},
f_{29}
\rangle=:J_5.
\end{eqnarray*}

As in Case $E_6$,
there exist $a, b$ such that
\begin{eqnarray*}
&&
\lim_{t\to 0}
\exp (t^{-1}\ad
(a X_{\substack{0011100\\ \phantom{00}0\phantom{1100}}}
+b X_{\substack{0110000\\ \phantom{01}1\phantom{0000}}}))
(J_5)\\
&=&
\langle X_\alpha\,|\, \hite(\alpha)\geq 24\rangle
\bigoplus
\langle
f_{23}, \Lambda^{(22)}
\rangle=K
\end{eqnarray*}
satisfies the condition in Proposition \ref{prop:JtoK}.

Then
\begin{eqnarray*}
\lim_{t\to 0}\Ad(\lambda_2^{-1}(t))
(K)
&=&
\fka_1,\\
\lim_{t\to 0}\exp(t^{-1}
\ad(X_{\substack{0100000\\ \phantom{01}0\phantom{0000}}}
))(K)
&=&
\fka_2.
\end{eqnarray*}
\end{proof}

\appendix

\section{The proof of Proposition \ref{prop:A:Inequality}}

First we treat two fundamental cases.

\begin{lemma}
\label{Case:(n-1)}
If $\mu=(n)$, then
$(z_1,\ldots,z_{n})=(1,1,\ldots,1,0)$ is a solution of $(IE_\mu)$.
\end{lemma}

\begin{proof}
In this case, $i(h)=n+1-h$ and $z_{i(h)}(h)=z_{n+1-h}+z_{n+2-h}+\cdots+z_{n}$.
Hence $(1,1,\ldots,1,0)$ is a solution.
\end{proof}

\begin{lemma}
\label{Case:(n-1)/2,(n-1)/2}
Suppose that $n$ is even and $\mu=(n/2, n/2)$.
Then
$$
\begin{array}{cccccccccc}
z_1 & z_2 & \cdots & z_{(n-2)/2} & z_{n/2} & z_{(n+2)/2} & z_{(n+4)/2} & \cdots & z_{n-1} & z_{n}\\
\hline
2 & 2 & \cdots & 2 & 0 & 3 & 2 & \cdots & 2 & 1
\end{array}
$$
is a solution of $(IE_\mu)$.
\end{lemma}

\begin{proof}
In this case, 
$$
i(h)=
\left\{
\begin{array}{ll}
\frac{n}{2}+1-h & (h\leq \frac{n}{2})\\
n+1-h & (h> \frac{n}{2})
\end{array}
\right.,
$$
and
$$
z_{i(h)}(h)=
\left\{
\begin{array}{ll}
\sum_{i=\frac{n}{2}+1-h}^{\frac{n}{2}}z_i & (h\leq \frac{n}{2})\\
\sum_{i=n+1-h}^{n}z_i & (h> \frac{n}{2})
\end{array}
\right..
$$
Hence the values of $z_i$ in the statement is a solution.
\end{proof}

\begin{lemma}
\label{KeyLemma}
For any $\mu\vdash n$, there exists a solution of 
the system $(IE_\mu)$.
\end{lemma}

\begin{proof}
We prove the assertion by induction on $n$.
For $n=1$, there is nothing to prove.

Let $\mu=(\mu_1\geq \mu_2\geq \cdots\geq \mu_l)\vdash n$.

\bigskip\noindent
(Case 1: The case $\mu_2<\mu_1$)

Define $\mu'=(\mu'_1\geq \mu'_2\geq \cdots\geq \mu'_l)\vdash n-1$ by
$$
\mu'_i:=
\left\{
\begin{array}{ll}
\mu_i & (i>1)\\
\mu_1-1 & (i=1).
\end{array}
\right.
$$
Let $z'=(z'_1,\ldots,z'_{n-1})$ be a solution of $(IE_{\mu'})$.
Then define $z=(z_1,\ldots,z_{n})$ by
$$
z_i:=
\left\{
\begin{array}{ll}
\sum_{j=1}^{n-1}z'_j & (i=1)\\
z'_{i-1} & (i>1).
\end{array}
\right.
$$
We claim that $z=(z_1,\ldots,z_{n})$ is a solution of 
the system $(IE_\mu)$.

Note that $i(h)=1$ if and only if $h=n$, since $\mu_2<\mu_1$.
For $h<n$, we have $i(h)'=i(h)-1$.
Hence, for $h<n$,
$$
z_{i(h)}(h)=z_{i(h)'+1}(h)
=z'_{i(h)'}(h)< z'_j(h)
=z_{j+1}(h)
$$
for $1\leq j\neq i(h)'$, or equivalently for $2\leq j+1\neq i(h)$.

By the definition of $z_1$,
we see $z_{i(h)}(h)<z_1(h)$ for $h<n$.
Since there is no condition for $h=n$ in $(IE_\mu)$,
we have proved that $z=(z_1,\ldots,z_{n})$ is a solution of 
the system $(IE_\mu)$.

\bigskip\noindent
(Case 2: The case $\mu_2=\mu_1=:m$ and $n+1\geq 3m$.)

Let $\tilde{\mu}:=(\mu_2\geq \mu_3\geq \cdots\geq \mu_l)$.
Then $\tilde{\mu}\vdash n-m$.
In this case, $\tilde{i}(h)=i(h)$ for $h\leq n-m$ by \eqref{eqn:i(h)}.
For $h>n-m$, we have $k=1$ in Lemma \ref{lem:i(h)},
and we have $i(h)=i(h-m)$ and $i(h)+h-1=n$ by \eqref{eqn:i(h)} and
\eqref{eqn:i(h)2}.

Let $\tilde{z}=(\tilde{z}_1,\ldots,\tilde{z}_{n-m})$ be a solution of 
$(IE_{\tilde{\mu}})$.
Define $z=(z_1,\ldots,z_{n})$ by
$$
z_i:=
\left\{
\begin{array}{ll}
\tilde{z}_i & (i\leq n-m)\\
\sum_{j=1}^{n-m}\tilde{z}_j & (i=n+1-m)\\
\tilde{z}_{i-m} & (i>n+1-m).
\end{array}
\right.
$$
Note that $i-m>m=\mu_1$ when $i>n+1-m$, since $n+1\geq 3m$.
Hence, in particular, $\tilde{z}_{i-m}>0$ for $i>n+1-m$.

We claim that $z=(z_1,\ldots,z_{n})$ is a solution of 
the system $(IE_\mu)$.
Note that
\begin{equation}
\label{eqn:n-m}
z_{i(h)}(h)<z_j(h)\quad
\text{if}\quad 
\left\{
\begin{array}{l}
n+1-m\in [j,j+h-1],\quad\text{and}\\
n+1-m\notin [i(h),i(h)+h-1]
\end{array}
\right.
\end{equation}
by the definition of $z_{n+1-m}$.
Note, also, that
$h< n+1-m$ implies 
$k\geq 2$ in Lemma \ref{lem:i(h)},
and hence
$i(h)+h-1\leq n-m$
by \eqref{eqn:i(h)2},
and thus $n+1-m\notin [i(h),i(h)+h-1]$.

Suppose that $h\leq m$. 
Then $i(h)+h-1\leq n-m$, since $h\leq m<n+1-m$.
Hence, if $n+1-m\in [j,j+h-1]$, or equivalently $n-m-h+2\leq j\leq n+1-m$, then
$z_{i(h)}(h)<z_j(h)$ by \eqref{eqn:n-m}.
If $j\leq n+1-m-h$, then, since $j+h-1<n+1-m$, we have
$$
z_j(h)=\tilde{z}_j(h)> \tilde{z}_{i(h)}(h)=z_{i(h)}(h).
$$
If $j>n+1-m$, then
$$
z_j(h)=\tilde{z}_{j-m}(h)> \tilde{z}_{i(h)}(h)=z_{i(h)}(h).
$$

Suppose that $m<h\leq n-m$.
If $n+1-m\notin [j,j+h-1]$, then by $j+h-1\leq n$ we have
$j+h-1<n+1-m$, and hence
$$
z_j(h)=\tilde{z}_j(h)> \tilde{z}_{i(h)}(h)=z_{i(h)}(h).
$$
Suppose that $n+1-m\in [j,j+h-1]$.
If $n+1-m\notin [i(h),i(h)+h-1]$,
then
$z_{i(h)}(h)<z_j(h)$ by \eqref{eqn:n-m}.
If $n+1-m\in [i(h),i(h)+h-1]$, then
$i(h)+h-1=n$ and $h\geq n+1-m$ by the definition of $i(h)$ 
(Lemma \ref{lem:i(h)}).

Suppose that $h\geq n-m+1$.
Then we have $i(h)=i(h-m)$ and $i(h)+h-1=n$.
Hence
\begin{eqnarray*}
z_j(h)-z_{i(h)}(h)
&=&
\sum_{k=j}^{i(h)-1}z_k-\sum_{k=j+h}^{n}z_k\\
&=&
\sum_{k=j}^{i(h)-1}\tilde{z}_k-\sum_{k=j+h}^{n}\tilde{z}_{k-m}\\
&=&
\sum_{k=j}^{i(h)-1}\tilde{z}_k-\sum_{k=j+h-m}^{n-m}\tilde{z}_{k}\\
&=&
\tilde{z}_j(h-m)-\tilde{z}_{i(h-m)}(h-m)>0.
\end{eqnarray*}
In the last equation, note that $i(h)<j+h-m$, 
since $i(h)\leq \mu_1=m$ and $j+h>n+1-m\geq 2m$.

\bigskip\noindent
(Case 3: The case $\mu_2=\mu_1=:m$ and $n+1< 3m$.)

Note that $n\geq 2m$, or equivalently $n+1-m>m$.
Let $\tilde{\mu}:=(\mu_2\geq \mu_3\geq \cdots\geq \mu_l)$.
Then $\tilde{\mu}\vdash n-m$.
Let $\tilde{z}=(\tilde{z}_1,\ldots,\tilde{z}_{n-m})$ be a solution of 
$(IE_{\tilde{\mu}})$.
Take $\delta\in \Q$ so that
$$
0< \delta < \min_{j\neq i(h)}(\tilde{z}_j(h)-\tilde{z}_{i(h)}(h)).
$$
Define $z=(z_1,\ldots,z_{n})$ by
$$
z_i:=
\left\{
\begin{array}{ll}
\tilde{z}_i & (i\leq n-m)\\
\sum_{j=1}^{n-m}\tilde{z}_j & (i=n+1-m)\\
\delta & (i=2m)\\
\tilde{z}_{i-m} & (i>n+1-m,\, i\neq 2m).
\end{array}
\right.
$$
The proof of the claim that
$z=(z_1,\ldots,z_{n})$ is a solution of 
the system $(IE_\mu)$
goes in the same way as in Case 2, except the case when
$h\geq n+1-m$, $j+h\leq 2m$, and $i(h)+h-1=n$.
In this case, we have
\begin{eqnarray*}
&&z_j(h)-z_{i(h)}(h)\\
&=&
z_j+\cdots+z_{i(h)-1}
-(z_{j+h}+\cdots+z_n)\\
&=&
\tilde{z}_j+\cdots+\tilde{z}_{i(h)-1}
-(\tilde{z}_{j+h-m}+\cdots+\tilde{z}_{n-m})
+\tilde{z}_m-\delta\\
&=&
\tilde{z}_j(h-m)-\tilde{z}_{i(h-m)}(h-m)-\delta
+\tilde{z}_m
>0.
\end{eqnarray*}
\end{proof}

\section{The proof of Proposition \ref{prop:JtoK} (Type $B$)}

In this section, let $\fkg:=\fks\fko(2n+1, \C)$ (cf. Example \ref{B-1}).
$$
\fks\fko(2n+1, \C)
=\left\{
\begin{bmatrix}
A & \x & B\\
-{}^t\y & 0 & -{}^t\x \\
C & \y & -A'
\end{bmatrix}
\,|\,
B'=-B,\, C'=-C
\right\}.
$$

Recall that
$$
\Lambda=
\sum_{i=1}^{n} E_{i, i+1} -\sum_{i=n+1}^{2n}E_{i,i+1}
=\sum_{i=1}^n\widetilde{E}_{i,i+1},
$$
and
$$
J=\bigoplus_{k=1}^{n}
\C \Lambda^{2k-1},
$$
where $\widetilde{E}_{i,j}=E_{i,j}-E_{2n+2-j,2n+2-i}\in \fkg$ and
\begin{eqnarray*}
\Lambda^{2k-1}
&=&
\sum_{i\leq n+1-(2k-1)}E_{i,i+2k-1}
+\!\!
\sum_{i=\max\{ n-2k+3, 1\}}^{\min\{ n,2n+1-(2k-1)\}}
\!\! (-1)^{n-i}E_{i,i+2k-1}\\
&&
-\sum_{i=n+1}^{2n+1-(2k-1)}E_{i,i+2k-1}\\
&=&
\sum_{i\leq n+1-(2k-1)}\widetilde{E}_{i,i+2k-1}
+
\sum_{i=n-2k+3}^{n+1-k}(-1)^{n-i}\widetilde{E}_{i,i+2k-1}.
\end{eqnarray*}

First suppose that $n$ is odd.
Let $n=2m+1$.
Put
\begin{equation}
\label{eqn:Bodd}
S:=\sum_{i=1}^{m+1}a_i
\widetilde{E}_{i,n+i}.
\end{equation}
Note that the height of $S$ is $n$.
By a simple computation, for $k\leq m$, 
\begin{eqnarray*}
[S,\Lambda^{2k-1}]
&=&
\sum_{i=1}^{m-2k+2}
(-a_i-a_{i+2k-1})\widetilde{E}_{i,i+n+2k-1}\\
&&
+\sum_{i=m-2k+3}^{m-k+1}(a_{n-i-2k+3}-a_i)
\widetilde{E}_{i,i+n+2k-1}.
\end{eqnarray*}

\begin{proof}[Proof of Proposition \ref{prop:JtoK}, odd $n$ case]
Take, for example, $a_i=i$ in \eqref{eqn:Bodd}. Then
\begin{eqnarray*}
[S,\Lambda^{2k-1}]
&=&\label{eqn:B:odd}
\sum_{i=1}^{m-2k+2}
-(2i+2k-1)\widetilde{E}_{i,i+n+2k-1}\\
&&
+\sum_{i=m-2k+3}^{m-k+1}(n+3-2k-2i)
\widetilde{E}_{i,i+n+2k-1}.
\end{eqnarray*}
By the consideration of height,
we have
$[S, [S, \Lambda^{2k-1}]]=0$ for all $k$
and
$[S, \Lambda^{2k-1}]=0$ for $2k-1\geq n$.
Hence
$$
\lim_{t\to 0}\exp t^{-1}\ad S(\C\Lambda^{2k-1})
=
\left\{
\begin{array}{ll}
\C [S, \Lambda^{2k-1}] & (k\leq m)\\
\C\Lambda^{2k-1} & (k>m),
\end{array}
\right.
$$
and
$K:=\langle 
\Lambda^{(l)}\,|\, n\leq l\leq 2n-1
\rangle$
with
$$
\Lambda^{(l)}
:=\left\{
\begin{array}{ll}
[S, \Lambda^{l-n}] & (\text{$l$: even})\\
\Lambda^l & (\text{$l$: odd, $l\geq n$})
\end{array}
\right.
$$
satisfies
$K=\lim_{t\to 0}\exp t^{-1}\ad S(J)$
and the condition of Proposition \ref{prop:JtoK}.
\end{proof}

Next suppose that $n$ is even.
Let $n=2m$.
Put
\begin{equation}
\label{eqn:Beven}
S:=\sum_{i=1}^{m+1}a_i \widetilde{E}_{i,i+n-1}.
\end{equation}
Note that the height of $S$ is $n-1$.
By a simple computation, for $k\leq m$,
\begin{eqnarray*}
[S,\Lambda^{2k-1}]
&=&\nonumber
(a_1+a_{2k})\widetilde{E}_{1,n+2k-1}\\
&&
-\sum_{i=2}^{m-2k+2}(a_i+a_{i+2k-1}) \widetilde{E}_{i,i+n+2k-2}\\
&&
-\sum_{i=m-2k+3}^{m-k+1}(a_i-a_{n-i-2k+4})\widetilde{E}_{i,i+n+2k-2}.
\end{eqnarray*}

\begin{proof}[Proof of Proposition \ref{prop:JtoK}, even $n$ case]
Take, for example, $a_i=i$ in \eqref{eqn:Beven}. Then
\begin{eqnarray*}
[S,\Lambda^{2k-1}]
&=&\nonumber
(2k+1)\widetilde{E}_{1,n+2k-1}\\
&&\label{eqn:B:even}
-\sum_{i=2}^{m-2k+2}(2i+2k-1) \widetilde{E}_{i,i+n+2k-2}\\
&&
+\sum_{i=m-2k+3}^{m-k+1}(n+4-2k-2i)\widetilde{E}_{i,i+n+2k-2}.
\end{eqnarray*}
By the consideration of height,
we have
$[S, [S, \Lambda^{2k-1}]]=0$ for all $k\geq 2$,
$[S, [S, \Lambda]]\in \C \Lambda^{2n-1}$,
and
$[S, \Lambda^{2k-1}]=0$ for $2k-1\geq n+1$.
Hence
by Lemma \ref{UnipotentDeformation}
\begin{eqnarray*}
&&[S, \Lambda^{2k-1}]\in \lim_{t\to 0}\exp t^{-1}\ad S(J)
\qquad (2k-1<n+1),\\
&&\Lambda^{2k-1}\in \lim_{t\to 0}\exp t^{-1}\ad S(J)\qquad (2k-1\geq n+1),
\end{eqnarray*}
and
$K:=\langle 
\Lambda^{(l)}\,|\, n\leq l\leq 2n-1
\rangle$
with
$$
\Lambda^{(l)}
:=\left\{
\begin{array}{ll}
[S, \Lambda^{l-(n-1)}] & (\text{$l$: even})\\
\Lambda^l & (\text{$l$: odd, $l> n$})
\end{array}
\right.
$$
satisfies
$K=\lim_{t\to 0}\exp t^{-1}\ad S(J)$
and the condition of 
Proposition \ref{prop:JtoK}.
\end{proof}

\section{The proof of Proposition \ref{prop:JtoK} (Type $C$)}

In this section, let $\fkg:=\fks\fkp(2n, \C)$ (cf. Example \ref{C-1}).
$$
\fks\fkp(2n, \C)
=\left\{
\begin{bmatrix}
A & B\\
C & -A'
\end{bmatrix}
\,|\,
B'=B,\, C'=C
\right\}.
$$

Recall that
$\Lambda=
\sum_{i=1}^n E_{i, i+1} -\sum_{i=n+1}^{2n-1}E_{i,i+1}
$,
and
$$
J=\bigoplus_{k=1}^{n}
\C \Lambda^{2k-1},
$$
where 
\begin{eqnarray*}
\Lambda^{2k-1}
&=&
\sum_{i\leq n-(2k-1)}E_{i,i+2k-1}
+
\sum_{i=\max\{ n-(2k-2), 1\}}^{\min\{ n,2n-(2k-1)\}}(-1)^{n-i}E_{i,i+2k-1}\\
&&
-\sum_{i=n+1}^{2n-(2k-1)}E_{i,i+2k-1}.
\end{eqnarray*}

Define an abelian Lie subalgebra
$K$ as follows:
$$
K:=
(\bigoplus_{k=\frac{n+1}{2}}^n
\C \Lambda^{2k-1})
\bigoplus
(\bigoplus_{k=1}^{\frac{n-1}{2}}
\C
\begin{bmatrix}
O & \Lambda_A^{2k-1}\\
O & O
\end{bmatrix})
\quad\text{if $n$ is odd,}
$$
$$
K:=
(\bigoplus_{k=\frac{n}{2}+1}^n
\C \Lambda^{2k-1})
\bigoplus
(\bigoplus_{k=1}^{\frac{n}{2}}
\C
\begin{bmatrix}
O & \Lambda_A^{2k-2}\\
O & O
\end{bmatrix})\quad\text{if $n$ is even,}
$$
where
$\Lambda_A:=\sum_{i=1}^{n-1}E_{i,i+1}$.
Put
$$
S:=
\begin{bmatrix}
O & -\frac{1}{2}I\\
O & O
\end{bmatrix}
\in \fkg
\quad \text{if $n$ is odd, and}
$$
$$
S:=
\begin{bmatrix}
\frac{1}{2}E_{1,n} & -\frac{1}{2}\sum_{i=1}^{n-1}E_{i+1,i}\\
O & -\frac{1}{2}E_{1,n}
\end{bmatrix}
\in \fkg
\quad \text{if $n$ is even.}
$$
Note that the height of $S$ is $n$ for $n$ odd and $n-1$ for $n$ even.

To prove
\begin{equation}
\label{eqn:CtypeJtoK}
K=\lim_{t\to 0}
\exp( t^{-1}\ad S)(J),
\end{equation}
we prepare three lemmas.
The following lemma is clear:

\begin{lemma}
\label{lemma:EZequation}
Suppose that $[A,C]=O$.
Then
$$
\left[
\begin{bmatrix}
A & B\\
O & -A'
\end{bmatrix},
\begin{bmatrix}
C & D\\
O & -C'
\end{bmatrix}
\right]
=
\begin{bmatrix}
O & AD-BC'-CB+DA'\\
O & O
\end{bmatrix}.
$$
\end{lemma}

For $X\in \fks\fkp(2n,\C)$, write
$$
X=
\begin{bmatrix}
X(11) & X(12)\\
X(21) & X(22)
\end{bmatrix}.
$$
Then
\begin{eqnarray}
\label{eqn:Ctype2k-1}
\Lambda^{2k-1}(11)&=&-\Lambda^{2k-1}(22)
=
\sum_{i=1}^{n-(2k-1)}E_{i,i+2k-1}=\Lambda_A^{2k-1},\\
\Lambda^{2k-1}(12)&=&
\sum_{i=\max\{ n-(2k-2),1\}}^{\min\{ n, 2n-(2k-1)\}}(-1)^{n-i}E_{i,i+2k-1-n},\nonumber\\
\Lambda^{2k-1}(21)&=&O.\nonumber
\end{eqnarray}

\begin{lemma}
\label{lemma:S:n:odd}
Suppose that $n$ is odd.
Then
$$
[S, \Lambda^{2k-1}]
=
\begin{bmatrix}
O & \Lambda_A^{2k-1}\\
O & O
\end{bmatrix}.
$$
\end{lemma}

\begin{proof}
In this case, $\displaystyle S=
\begin{bmatrix}
O & -\frac{1}{2}I\\
O & O
\end{bmatrix}
$.
Hence the assertion is clear from Lemma \ref{lemma:EZequation} and \eqref{eqn:Ctype2k-1}.
\end{proof}

\begin{lemma}
\label{lemma:S:n:even}
Suppose that $n$ is even.
Then
$$
[S, \Lambda^{2k-1}]
=
\begin{bmatrix}
O & \Lambda_A^{2k-2}\\
O & O
\end{bmatrix}.
$$
\end{lemma}

\begin{proof}
In this case,
$S(11)=-S(22)=\frac{1}{2}E_{1n}$, $S(21)=O$, and
$
S(12)=-\frac{1}{2}\sum_{i=1}^{n-1}E_{i+1,i}.
$
Note that $[S(11),\Lambda^{2k-1}(11)]=O$.
By Lemma \ref{lemma:EZequation},
$[S, \Lambda^{2k-1}](11)=[S, \Lambda^{2k-1}](22)
=[S, \Lambda^{2k-1}](21)=O$, and
\begin{eqnarray*}
[S, \Lambda^{2k-1}](12)
&=&
S(11)\Lambda^{2k-1}(12)-S(12)\Lambda^{2k-1}(11)\\
&&\qquad
-\Lambda^{2k-1}(11)S(12)+\Lambda^{2k-1}(12)S(11).
\end{eqnarray*}
If $2k-1>n$, then we have $[S, \Lambda^{2k-1}]=O$ by the consideration of height, and hence 
the assertion holds.

Suppose that $2k-1<n$, then
$$
\Lambda^{2k-1}(12)=
\sum_{i=n-(2k-2)}^{ n}(-1)^{n-i}E_{i,i+2k-1-n}
$$
\begin{eqnarray*}
S(11)\Lambda^{2k-1}(12)
&=&
\frac{1}{2}E_{1,n}\sum_{i=n-2k+2}^{n}(-1)^{n+i}E_{i,i+2k-1-n}
=\frac{1}{2}E_{1,2k-1},\\
\Lambda^{2k-1}(12)S(11)
&=&
\frac{1}{2}\sum_{i=n-2k+2}^{n}(-1)^{n+i}E_{i, i+2k-1-n}E_{1,n}
=
\frac{1}{2}E_{n-2k+2,n},\\
S(12)\Lambda^{2k-1}(11)
&=&
-\frac{1}{2}(\sum_{j=1}^{n-1}E_{j+1,j})
(\sum_{i=1}^{n-2k+1}E_{i,i+2k-1})\\
&=&
-\frac{1}{2}\sum_{i=1}^{n-2k+1}E_{i+1,i+2k-1}
=
-\frac{1}{2}\sum_{i=2}^{n-2k+2}E_{i,i+2k-2},\\
\Lambda^{2k-1}(11)S(12)
&=&
-\frac{1}{2}(\sum_{i=1}^{n-2k+1}E_{i,i+2k-1})
(\sum_{j=1}^{n-1}E_{j+1,j})\\
&=&
-\frac{1}{2}\sum_{i=1}^{n-2k+1}E_{i,i+2k-2}.
\end{eqnarray*}
Hence
$$
[S, \Lambda^{2k-1}](12)=
\sum_{i=1}^{n-2k+2}E_{i,i+2k-2}
=\Lambda_A^{2k-2}.
$$
\end{proof}

\begin{proof}[Proof of Proposition \ref{prop:JtoK}]
By Lemma \ref{lemma:S:n:odd},
$$
\lim_{t\to 0}
\exp( t^{-1}\ad S)(\C \Lambda^{2k-1})
=
\left\{
\begin{array}{ll}
\C \Lambda^{2k-1} & (2k-1\geq n)\\
\C
\begin{bmatrix}
O & \Lambda_A^{2k-1}\\
O & O
\end{bmatrix}
& (2k-1<n)
\end{array}
\right.
$$
for $n$ odd.
By Lemma \ref{lemma:S:n:even},
$$
\lim_{t\to 0}
\exp( t^{-1}\ad S)(\C \Lambda^{2k-1})
=
\left\{
\begin{array}{ll}
\C \Lambda^{2k-1} & (2k-1> n)\\
\C
\begin{bmatrix}
O & \Lambda_A^{2k-2}\\
O & O
\end{bmatrix}
& (2k-1<n)
\end{array}
\right.
$$
for $n$ even.
Hence 
$$
K=\lim_{t\to 0}
\exp( t^{-1}\ad S)(J).
$$
It is clear that $K$ satisfies the condition in  Proposition \ref{prop:JtoK}.
\end{proof}

\section{The proof of Proposition \ref{prop:JtoK} (Type $D$)}
\label{appendixD}

In this section, let $\fkg:=\fks\fko(2n, \C)$ (cf. Example \ref{D-1}).
$$
\fks\fko(2n, \C)
=\left\{
\begin{bmatrix}
A & B\\
C & -A'
\end{bmatrix}
\,|\,
B'=-B,\, C'=-C
\right\}.
$$

Put
$$
\widetilde{E}_{i,j}=E_{i,j}-E_{2n+1-j,2n+1-i}.
$$
Recall that
\begin{eqnarray*}
\Lambda&=&
\sum_{i=1}^{n-1} E_{i, i+1} -\sum_{i=n+1}^{2n-1}E_{i,i+1}
+E_{n-1, n+1}-E_{n,n+2}\\
&=&
\sum_{i=1}^{n-1}\widetilde{E}_{i,i+1}
+\widetilde{E}_{n-1,n+1},
\end{eqnarray*}
and
$$
J=\C Z\bigoplus\bigoplus_{k=1}^{n-1}
\C \Lambda^{2k-1},
$$
where 
\begin{eqnarray*}
Z&=&E_{1,n}-E_{n+1,2n}-E_{1,n+1}+E_{n,2n}
=
\widetilde{E}_{1,n}-\widetilde{E}_{1,n+1}.
\end{eqnarray*}
We have
\begin{eqnarray*}
\Lambda^{2k-1}
&=&
\sum_{i=1}^{n-(2k-1)}
\widetilde{E}_{i,i+2k-1}+\widetilde{E}_{n-(2k-1),n+1}\\
&&\qquad
+2\sum_{i=1}^{k-1}(-1)^i
\widetilde{E}_{n-(2k-1)+i,n+1+i}.
\end{eqnarray*}
The height of $\Lambda^{2k-1}$ equals $2k-1$, and
that of $Z$ $n-1$.
Note that, when $n$ is even,
$$
\Lambda^{n-1}=
\widetilde{E}_{1,n}+\widetilde{E}_{1,n+1}
+2\sum_{i=1}^{\frac{n}{2}-1}(-1)^i
\widetilde{E}_{1+i,n+1+i}.
$$


Let $1\leq i<n,\, 1<j< 2n,\, i<j,\, i< 2n+1-j$.
Then
the height of $\widetilde{E}_{i,j}$ equals $j-i$ for $j\leq n$
and $j-i-1$ for $j>n$.
Hence, 
$\C \widetilde{E}_{1, n}$ and
$\C \widetilde{E}_{i,i+n}$ $( i<\frac{n+1}{2})$ are all the
root spaces of height $n-1$, and,
for $h\geq n$, 
$\C \widetilde{E}_{i,i+h+1}$ $( i<n-\frac{h}{2})$ are all the
root spaces of height $h$.

First suppose that $n$ is odd.
Put
\begin{equation}
\label{eqn:SDodd}
S:=\sum_{i=1}^{2} a_i
\widetilde{E}_{i,i+n-2}+
\sum_{i=2}^{\frac{n+1}{2}}b_i\widetilde{E}_{i,i+n-1}.
\end{equation}
By a simple computation,
\begin{eqnarray*}
[S, \Lambda]
&=&
(a_1-a_2)\widetilde{E}_{1,n}+ (a_1-b_2)\widetilde{E}_{1, n+1}\\
&&
\quad
-(a_2+b_2+b_3)\widetilde{E}_{2, n+2}
-\sum_{i=3}^{\frac{n-1}{2}}(b_i+b_{i+1})\widetilde{E}_{i, n+i},
\end{eqnarray*}
and for $k\geq 2$
\begin{eqnarray*}
[S, \Lambda^{2k-1}]
&=&
(2a_1-b_{2k})\widetilde{E}_{1, n+2k-1}
-(a_2+b_2+b_{2k+1})\widetilde{E}_{2, n+2k}\\
&&\quad
-\sum_{i=3}^{\frac{n+3-4k}{2}}
(b_i+b_{i+2k-1})\widetilde{E}_{i,i+n+2k-2}\\
&&
\qquad
-\sum_{i=\frac{n+5-4k}{2}}^{\frac{n-2k+1}{2}}
(b_i-b_{n+3-i-2k})\widetilde{E}_{i,i+n+2k-2}.
\end{eqnarray*}

\begin{proof}[Proof of Proposition \ref{prop:JtoK}, odd $n$ case]
Take, for example, $a_1=2, a_2=1, b_2=1$, and
$b_i=-i$ $(i\geq 3)$ in \eqref{eqn:SDodd}.
Then
$$
[S, \Lambda]=
\widetilde{E}_{1, n}+\widetilde{E}_{1, n+1}+\widetilde{E}_{2, n+2}
+\sum_{i=3}^{\frac{n-1}{2}}(2i+1)\widetilde{E}_{i, n+i},
$$
and for $k\geq 2$
\begin{eqnarray*}
[S, \Lambda^{2k-1}]
&=&
(4+2k)\widetilde{E}_{1, n+2k-1}
+(2k-1)\widetilde{E}_{2, n+2k}\\
&&\quad
+\sum_{i=3}^{\frac{n+3-4k}{2}}
(2i+2k-1)\widetilde{E}_{i, i+n+2k-2}\\
&&\qquad
-
\sum_{i=\frac{n+5-4k}{2}}^{\frac{n+1-2k}{2}}
(n+3-2k-2i)\widetilde{E}_{i,i+n+2k-2}.
\end{eqnarray*}
By the consideration of height, we have
$$
[S, \Lambda^{2k-1}]=0 \qquad \text{for $2k-1\geq n$}
$$
and
$$
[S, Z]\in \C \Lambda^{2n-3}.
$$
Hence
$$
\lim_{t\to 0}\exp t^{-1}\ad S(J)
=\langle Z, 
\Lambda^{(l)}\,|\, n-1\leq l\leq 2n-3
\rangle,
$$
where
$$
\Lambda^{(l)}
=\left\{
\begin{array}{ll}
[S, \Lambda^{l-(n-2)}] & (\text{$l$: even})\\
\Lambda^l & (\text{$l$: odd, $l\geq n$}).
\end{array}
\right.
$$
\end{proof}

Next suppose that $n$ is even.
Let
\begin{equation}
\label{eqn:SDeven}
S:= a\widetilde{E}_{1,n}
+\sum_{i=1}^{\frac{n}{2}}b_i \widetilde{E}_{i,i+n}.
\end{equation}

By a simple computation,
\begin{eqnarray*}
[S, \Lambda^{2k-1}]
&=&
-(a+b_1+b_{2k})\widetilde{E}_{1,n+2k}
-\sum_{i=2}^{\frac{n+2-4k}{2}}
(b_i+b_{i+2k-1})
\widetilde{E}_{i,i+n+2k-1}\\
&&
\qquad
-\sum_{i=\frac{n+4-4k}{2}}^{\frac{n-2k}{2}}
(b_{i}
-
b_{n+2-2k-i})
\widetilde{E}_{i,i+n+2k-1}.
\end{eqnarray*}

\begin{proof}[Proof of Proposition \ref{prop:JtoK}, even $n$ case]

Take, for example, $a=0$ and
$b_i=-i$ in \eqref{eqn:SDeven}. Then
\begin{eqnarray*}
[S, \Lambda^{2k-1}]
&=&
(2k+1)\widetilde{E}_{1,n+2k}
+\sum_{i=2}^{\frac{n+2-4k}{2}}
(2i+2k-1)
\widetilde{E}_{i,i+n+2k-1}\\
&&
\qquad
-\sum_{i=\frac{n+4-4k}{2}}^{\frac{n-2k}{2}}
(n+2-2k-2i)
\widetilde{E}_{i,i+n+2k-1}.
\end{eqnarray*}
By the consideration of height, we have
$$
[S, \Lambda^{2k-1}]=0 \qquad \text{for $2k\geq n$}
$$
and
$$
[S, Z]=0.
$$
Hence
$$
\lim_{t\to 0}\exp t^{-1}\ad S(J)
=\langle Z, 
\Lambda^{(l)}\,|\, n-1\leq l\leq 2n-3
\rangle,
$$
where
$$
\Lambda^{(l)}
=\left\{
\begin{array}{ll}
[S, \Lambda^{l-(n-1)}] & (\text{$l$: even})\\
\Lambda^l & (\text{$l$: odd, $l\geq n-1$}).
\end{array}
\right.
$$
\end{proof}


\end{document}